\documentclass[10pt]{amsart}
\usepackage[utf8]{inputenc}

\usepackage{amssymb,amsmath,amsthm}
\usepackage{mathdots}
\usepackage{xcolor}
\usepackage{enumitem}
\usepackage{comment}
\usepackage{nicematrix}
\usepackage{tikz}
\usepackage{upgreek}
\usepackage[colorlinks=true,citecolor=blue,linkcolor=blue,urlcolor=blue]{hyperref}

\newtheorem{thm}{Theorem}[section]
\newtheorem{cor}[thm]{Corollary}
\newtheorem{lem}[thm]{Lemma}
\newtheorem{prop}[thm]{Proposition}

\theoremstyle{definition}
\newtheorem{rem}[thm]{Remark}
\newtheorem{ques}[thm]{Question}
\newtheorem{prob}[thm]{Problem}
\newtheorem{conj}[thm]{Conjecture}
\newtheorem{exam}[thm]{Example}

\newcommand{\sgn}{\operatorname{sgn}}

\newcommand{\C}{\mathbb{C}}
\newcommand{\N}{\mathbb{N}}
\newcommand{\Z}{\mathbb{Z}}
\newcommand{\Q}{\mathbb{Q}}

\newcommand{\cal}{\mathcal}

\title{Hankel determinants of weighted binary sums of digits}
\author{Bartosz Sobolewski and  Maciej Ulas}


\thanks{The research of the authors was supported by a grant of the National Science Centre (NCN), Poland, no. UMO-2019/34/E/ST1/00094}

\begin{document}

\begin{abstract}
Let $s_\mathbf{w}$ be the weighted binary sum-of-digits function associated with an arbitrary sequence of complex weights $\mathbf{w}=(w_j)_{j\geq 0}$.
We investigate Hankel determinants $\mathcal{H}_\mathbf{w}(n) = \det [s_{\mathbf{w}}(i+j)]_{0\leq i,j<n}$ and derive a general recursion that allows us to effectively compute $\mathcal{H}_\mathbf{w}(n)$ for all $n$. Applying it to the ordinary binary sum-of-digits, that is, $w_j=1$, we express $\mathcal{H}_\mathbf{w}(n)$ in a closed form for several sequences of indices, including the remarkably simple
$$
\mathcal{H}_\mathbf{w}(\lceil 2^{k+2}/3\rceil)=
(-1)^{\frac{(k+2)(k+3)}{2}}(k+1).
$$
This yields an infinite family of explicit evaluations, giving a partial solution to a problem posed by Allouche and Shallit.
Moreover, we closely study the specialization $w_j=t^j$, where the determinants become polynomials in $t$, and investigate their vanishing. For $t=2\zeta$, where $\zeta$ is a root of unity, we show that the determinants vanish on a large structured set of indices, while the complementary is sparse but infinite. 

In addition to $\mathcal{H}_\mathbf{w}(n)$, we consider Hankel determinants associated with the first difference of $s_{\mathbf{w}}$, obtaining an explicit product formula. This generalizes the results by Fokkink, Kraaikamp, and Shallit concerning Hankel determinants for the period-doubling sequence.
\end{abstract}

\maketitle

\section{Introduction} \label{sec:Intro}
By the $n$th
\textit{Hankel matrix} of a sequence $(a_{n})_{n \geq 0}$, we mean a matrix of the form
\[
H(n)=\bigl[a_{i+j}\bigr]_{0\le i,j<n}.
\]
Hankel matrices and the corresponding \textit{Hankel determinants}
$\det H(n)$ occur naturally in several classical contexts, including moment problems, continued fractions, Pad\'{e} approximation, orthogonal polynomials, automata theory, and Diophantine approximation.  For instance, a real sequence $(a_n)_{n\geq 0}$ is the moment sequence of a positive Borel measure $\mu$ on the real line if and only if all Hankel matrices are positive semi-definite. This is a solution of the classical Hamburger moment problem on the real line. Moreover, Hankel determinants control the existence and norms of the associated monic orthogonal polynomials. Thus, in this classical setting, the sequence of Hankel determinants encodes important
structural information about the underlying moment sequence. Another application of Hankel determinants was developed by Bugeaud, Han, Wen, and Yao~\cite{BHWY}, who used them as an effective method in the study of irrationality exponents of automatic and related numbers. Their approach uses Pad\'{e} approximants constructed from generating functions, with the
non-vanishing of suitable Hankel determinants guaranteeing approximants of the required order. In this way, Hankel determinants connect the combinatorial structure of automatic sequences with sharp Diophantine approximation results, including cases where direct control of the continued fraction expansion is
unavailable.

In the present paper, we pursue the line of research concerning the evaluation of Hankel determinants associated with digital and automatic sequences, which was initiated by the work of Allouche, Peyri\`{e}re, Wen and Wen~\cite{APWW}. They proved non-vanishing of Hankel determinants associated with the Thue--Morse sequence, as well as the $2$-automaticity of a related two-dimensional sequence of determinants (reduced modulo $2$). The former result was later proved using combinatorial methods by Bugeaud and Han~\cite{BH}, while the latter was extended by Guo and Wen~\cite{GW} to $k$th order differences of the Thue--Morse sequence. 

Han~\cite{H1} developed ``Hankel determinant calculus'', which supplies explicit evaluation techniques (modulo a suitable integer) for several classes of automatic and Thue--Morse-like sequences, and later introduced the notion of \emph{Hankel continued fractions}~\cite{H2}. The latter work connects the sequence of Hankel determinants of a formal power series with specific continued fraction expansions and yields, in many situations, periodicity and algebraic descriptions of the associated determinants. These methods have become useful
tools for handling Hankel determinants of sequences satisfying self-similar functional equations or substitution rules. Another important structural result was obtained by Hu and Han~\cite{HH}, who proved that for a large class of $\pm1$-valued automatic sequences defined through infinite product representations, the Hankel determinant sequence reduced modulo~$2$ is itself automatic. Thus, despite the complexity of computing the exact values of Hankel determinants, the resulting sequence retains the underlying regularity of the original automatic sequence. This phenomenon supports the viewpoint that Hankel determinants often reflect the intrinsic automaton-theoretic structure.
Beyond the setting of automatic sequences, we note that the classical determinant calculus developed by Krattenthaler~\cite{K1,K2} is a valuable computational tool.

Exact evaluations of the determinants $\det H(n)$ are rare in this area, as most results focus on their behavior modulo an integer. A notable exception is the work of Fokkink, Kraaikamp, and Shallit~\cite{FKS}. They computed closed expressions for the Hankel determinants of the period-doubling sequence, revealing a link to Jacobsthal numbers and obtaining the precise recursive structure underlying the determinants. 

\subsection{The present work}

The fundamental goal of the present study is to proceed in this direction of exactly evaluating Hankel determinants for digital sequences.
Our starting point is the following problem found in the monograph by Allouche and Shallit \cite[p.\ 208, Problem 4]{AS} (worded slightly differently).

\begin{prob} \label{prob:AS03}
    Find a simple closed form for the determinants of Hankel matrices given by
    $H(n)=[s(i+j)]_{0 \leq i,j < n}$, where $s$ is the binary sum of digits     function.
\end{prob}

The sequence $(\det H(n))_{n \geq 1}$ is A056886 in the OEIS \cite{OEIS} and begins with
$$
0, -1, 2, 5, -8, 3, 8, 141, -204, 11, -4, -11, -30, 99, 1836, 164997, \ldots.
$$
A natural approach here is to establish a recursion involving $\det H(n)$, using the self-similarity of the sum-of-digits function. 
However, when performing matrix operations to reduce the size of the determinant, it becomes apparent that the contribution of digits at certain positions may be perturbed. Therefore, working with just the ordinary sum of digits becomes too restrictive, and it is (more or less) necessary to extend the scope to a more general family of functions, namely weighted binary sums of digits. More precisely, let $\mathbf{w}=(w_j)_{j \geq 0}$ be a sequence of complex weights. Then for a number $n \in \N$  with binary expansion $n = 2^l \varepsilon_l + \cdots + 2 \varepsilon_1 + \varepsilon_0$, where $\varepsilon_i \in \{0,1\}$, its weighted binary sum of digits $s_{\mathbf{w}}(n)$ is defined by
$$ s_{\mathbf{w}}(n) := w_l \varepsilon_l + \cdots + w_1 \varepsilon_1 + w_0 \varepsilon_0. $$
In particular, if all weights are equal to $1$, then $s_{\mathbf{w}}$ is the usual binary sum-of-digits function. For $n \geq 1$ we define the associated Hankel matrices
$$ H_{\mathbf{w}}(n) := [s_{\mathbf{w}}(i+j)]_{0 \leq i,j < n} $$
and determinants
$$ \mathcal{H}_{\mathbf{w}}(n) := \det  H_{\mathbf{w}}(n) .$$

The central result of this paper (Theorem \ref{thm:det_H_recurrence}) establishes a set of recursive formulas that allow us to effectively calculate exact values of $\mathcal{H}_\mathbf{w}(n)$ for all $n$. This already provides a partial solution to Problem \ref{prob:AS03} in a more general setting. We explore various consequences of this theorem, ranging from a general structural description of the determinants $\mathcal{H}_\mathbf{w}(n)$ to a detailed analysis of their behavior for special weights and indices. 

Although we do not obtain a single closed form for arbitrary $n$, in the case of the ordinary sum-of-digits function we can explicitly evaluate $\det H(n)$  for several infinite families of indices such as $n=n_k:=\lceil 2^{k+2}/3\rceil$, $n=2^k$, and $n=2^k\pm 1$. Most interestingly, in the first case a particularly simple formula holds:
\begin{equation} \label{eq:simple_expression}
    \det H(n_k) =  (-1)^{\frac{(k+2)(k+3)}{2}} (k+1).
\end{equation}

The behavior of $\mathcal{H}_{\mathbf{w}}(n)$ for polynomial weights $w_j=t^j$ turns out to be quite interesting as well. To distinguish this case from the generic one, for $n = 2^l \varepsilon_l + \cdots + 2 \varepsilon_1 + \varepsilon_0$ instead of $s_{\mathbf{w}}(n)$ we write $S(n,t)$, that is,
$$
S(n,t):=\sum_{j=0}^{l}\varepsilon_j t^j.
$$
We use the following special notation for related Hankel matrices and determinants:
\begin{align*}
    H(n,t) &:=[S(i+j,t)]_{0\leq i,j<n}, \\
    \mathcal{H}(n,t) &:=\det H(n,t).
\end{align*}
As a consequence of our main theorem, in Theorem \ref{thm:nice_formula} we obtain an explicit polynomial formula 
\begin{equation} \label{eq:polynomial_simple_expression}
    \mathcal{H}(n_k,t)=(-1)^{\frac{(k+1)(k+2)}{2}}t^{\alpha_k}(t-2)^{n_k-2}(1+t+\cdots+t^k),
\end{equation}
for certain $\alpha_k \in \N$, which directly implies \eqref{eq:simple_expression}. Other applications of Theorem \ref{thm:det_H_recurrence} include the determination of signs of $\det H(n)$ and the degrees $ \deg_t \mathcal{H}(n,t)$. 

A significant part of the paper is devoted to the investigation of zero sets of $\mathcal{H}(n,t)$, which is a natural problem in view of the earlier discussion. One of the main results is Theorem~\ref{thm:H_zero},
which shows that for every root of unity $\zeta$ of order $d\geq 2$, the determinants
$\mathcal{H}(n,2\zeta)$ vanish on long intervals centered at multiples of large powers of $2$. This is complemented by Theorem~\ref{thm:H_zero2}, which gives quantitative bounds for the number of exceptional indices $n$ for which the corresponding determinant does not vanish, in particular implying the sparsity of these sets. On the other hand, we exhibit an infinite set of indices, different from $n_k$, such that  $\mathcal{H}(n,2\zeta) \neq 0$ (Theorem \ref{thm:H_nonzero}).

Recursive formulas for $H_{\mathbf{w}}(n)$ involve another sequence of determinants, closely related to Hankel determinants $\mathcal{G}_{\mathbf{w}}(n)$ for the first-order difference
$$
d_{\mathbf{w}}(n) := s_{\mathbf{w}}(n+1) - s_{\mathbf{w}}(n)
$$
Despite playing an auxiliary role, these determinants turn out to be quite interesting in their own right.
In Theorem~\ref{thm:det_G} we derive a closed product form for $\mathcal{G}_{\mathbf{w}}(n)$, which is controlled by the run lengths in the binary expansion of $n$.
As a special case, by choosing weights to be Jacobsthal numbers, we recover the result by Fokkink, Kraaikamp and Shallit \cite[Theorem 2]{FKS} concerning Hankel determinants for the period-doubling sequence and its bit-wise negation. We note that there appears to be a sign error in its original statement --- a correct version is given in Proposition \ref{prop:G_period_doubling}. For the polynomial specialization $w_j=t^j$, in Corollary~\ref{cor:det_G(n,t)} we obtain an elegant formula for $\mathcal{G}_{\mathbf{w}}(n)$, giving a full description of the roots and their multiplicities.

The structure of the paper is as follows. Section \ref{sec:Notation} introduces the notation and conventions commonly used throughout the paper. In Section~\ref{sec:G} we focus on the determinants
$\mathcal{G}_{\mathbf{w}}(n)$ associated with the first-order difference $d_{\mathbf{w}}$. We provide the general product formula and the applications to special cases, including the period-doubling sequence and polynomial weights $w_j=t^j$.
In Section~\ref{sec:H} we turn to the determinants $\mathcal{H}_{\mathbf{w}}(n)$, associated with the
weighted binary sum-of-digits function. The main result is the recursive relation for $\mathcal{H}_{\mathbf{w}}(n)$, while the remainder of the section is devoted to its various implications.
Section~\ref{sec:vanishing} is focused on vanishing and non-vanishing phenomena for the polynomials $\mathcal{H}(n,t)$.
Finally, in Section \ref{sec:Questions} we collect several questions and problems suggested by the results obtained in the paper.

\section{Notation and conventions} \label{sec:Notation}
As we have already mentioned, we use the convention $\N = \{0,1,2,...\}$. Moreover, we let $\N_+ = \{1,2,\ldots  \}$ and, more generally, $\N_{\geq k} = \{k,k+1,k+2,\cdots \}$ for $k \in \N$.
For $n \in \N_+$ we will use the following notation for vectors and matrices that appear often throughout the paper:
\begin{itemize}
    \item $\mathbf{0}_n$ --- length $n$ column vector with all entries $0$;
    \item $\mathbf{1}_n$ --- length $n$ column vector with all entries $1$;
    \item  $\mathbf{0}$ --- matrix with all entries $0$ of any size (implied by the context);
    \item $J_n$ --- $n \times n$ matrix with all entries $1$, 
    \item $K_n$ --- $n \times n$ square matrix with $1$s below the antidiagonal and $0$s elsewhere;
    \item $L_n$ --- $n \times n$ square matrix with $1$s on the antidiagonal and $0$s elsewhere;
    \item $M_n$ --- $n \times n$ square matrix with $1$s on and below the antidiagonal and $0$s elsewhere.
\end{itemize}
More explicitly, the last three matrices are \vspace{1.5em}
$$
K_n = \begin{bmatrix}
    0 & 0 & \cdots   & 0 \\
    0 & &   \iddots  & 1 \\
    \vdots  & \iddots & \iddots & \vdots \\
    0 & 1 & \cdots & 1    
\end{bmatrix},
\quad
L_n = \begin{bmatrix}
    0 & \cdots  & 0 & 1 \\
    \vdots & & \iddots & 0 \\
    0 & \iddots  &  & \vdots \\
    1 & 0 & \cdots & 0
\end{bmatrix}, \quad
M_n = \begin{bmatrix}
    0 & 0 & \cdots   & 1 \\
    0 & &   \iddots  & 1 \\
    \vdots  & \iddots & \iddots & \vdots \\
    1 & 1 & \cdots & 1    
\end{bmatrix}.
$$
It is useful to note that
$$ \det L_n = \det M_n = (-1)^\frac{n(n-1)}{2} = \begin{cases}
    1 &\text{if } n \equiv 0,1 \pmod{4}, \\
    -1 &\text{if } n \equiv 2,3 \pmod{4}.
\end{cases} $$
We also use the standard notation $\otimes$ for the Kronecker product of matrices (see for example \cite[Chapter 20]{LM}). The property we are going to use the most is
$$ (A \otimes B)(C \otimes D) = (AB) \otimes (CD) $$
for matrices $A,B,C,D$ of appropriate size.

Below we list several further conventions:
\begin{itemize}
    \item the rows and columns of matrices are indexed from $0$;
    \item in a block matrix, non-bold $0$ or $1$ means a single entry;
    \item when writing block matrices, whose block sizes are symmetric along the diagonal, we sometimes suppress the indices of the matrices and write the heights of the blocks to the right of the matrix, for example:
    $$
    \begin{bNiceArray}{cccccc}
        0 & 0 & 0 & 0 & 0 & 0 \\
        0 & 0 & 0 & 1 & 0 & 0 \\
        0 & 0 & 1 & 1 & 0 & 1 \\
        0 & 1 & 1 & 1 & 1 & 0 \\
        0 & 1 & 0 & 0 & 1 & 1 \\
        1 & 1 & 0 & 0 & 1 & 1 
   \CodeAfter
\tikz \draw [dashed] (5-|1) -- (5-|last);
\tikz \draw [dashed] (3-|5) -- (3-|last);
\tikz \draw [dashed] (1-|5) -- (last-|5);
\tikz \draw [dashed] (5-|3) -- (last-|3);
    \end{bNiceArray} =
        \begin{NiceArray}{cccl}[left-margin=.6em,right-margin=.4em]
    \Block[borders={bottom,right,tikz=dashed}]{2-2}<\Large> K &  & \mathbf{0} & 2 \\  &  & L & 2 \\ M & \mathbf{0} & J &2
    \CodeAfter
\SubMatrix[{1-1}{3-3}]
\end{NiceArray};
$$
    \item when considering $H_{\mathbf{w}}(n), G_{\mathbf{w}}(n)$ for $2^k < n \leq 2^{k+1}$, we sometimes specify $\mathbf{w}$ as the finite sequence of weights that actually appear in the matrices, that is, $\mathbf{w}= (w_0,w_1,\ldots,w_{k+1})$.
\end{itemize}

\section{Formulas for the determinants $\mathcal{G}_{\mathbf{w}}(n)$} \label{sec:G}

In this section, we consider Hankel matrices and determinants associated with the first-difference sequence $(d_{\mathbf{w}}(n))_{n \geq 0}$, that is,
\begin{align*}
    G_{\mathbf{w}}(n) &:= [d_{\mathbf{w}}(i+j)]_{0\leq i,j < n}, \\
\mathcal{G}_{\mathbf{w}}(n) &:= \det G_{\mathbf{w}}(n).
\end{align*}
We derive general formulas for $\mathcal{G}_{\mathbf{w}}(n)$ and later apply them to special weights $\mathbf{w}$. In particular, we obtain as a corollary the aforementioned result from \cite{FKS} on Hankel determinants of the period-doubling sequence. 

We also consider related $n \times n$ matrices $\widetilde{G}_{\mathbf{w}}(n)$ and their determinants:
\begin{align*}
  \widetilde{G}_{\mathbf{w}}(n) &:= 
\begin{bmatrix}
    d_{\mathbf{w}}(0) &  \cdots & d_{\mathbf{w}}(n-2) & 1 \\
    \vdots &  & \vdots & \vdots \\
    d_{\mathbf{w}}(n-1) &  \cdots & d_{\mathbf{w}}(2n-3) & 1 \\
\end{bmatrix}, \\
\widetilde{\mathcal{G}}_{\mathbf{w}}(n) &:= \det \widetilde{G}_{\mathbf{w}}(n).
\end{align*}
In other words, $\widetilde{G}_{\mathbf{w}}(n)$ is obtained from $G_{\mathbf{w}}(n)$ by replacing all entries in the rightmost column with $1$s. Their determinants will appear in Section \ref{sec:H} in the recursion for $\mathcal{H}_{\mathbf{w}}(n)$.

\subsection{Recurrence relations}
To begin, we express in a more explicit way the entries $d_{\mathbf{w}}(m)$. If we write $m+1 =  2^{k+1} l + 2^k$ for some $k,l \in \N$, then
$$ d_{\mathbf{w}}(m) = s_{\mathbf{w}}(2^{k+1} l + 2^k)- s_{\mathbf{w}} \left(2^{k+1} l + \sum_{j=0}^{k-1} 2^j \right) = w_k - \sum_{j=0}^{k-1} w_j.$$
We can rewrite this equality in terms of the sequence $\mathbf{u}=(u_n)_{n \geq 0}$, defined by 
\begin{align*}
    u_0 &:= w_0, \\
u_n &:= w_n - 2w_{n-1}, \quad n \geq 1.
\end{align*} 
Then we have the expression
\begin{equation} \label{eq:d_w}
    d_{\mathbf{w}}(m) = \sum_{i=0}^{\nu_2(m+1)} u_i,
\end{equation}
where $\nu_2$ denotes the $2$-adic valuation.
For example, we have
$$ G_{\mathbf{w}}(4) = \begin{bmatrix}
    u_0 & u_0+u_1 & u_0 & u_0+u_1+u_2   \\
    u_0+u_1 & u_0 &  u_0+u_1+u_2 & u_0   \\
    u_0 & u_0+u_1+u_2& u_0 &  u_0+u_1  \\
   u_0+u_1+u_2 & u_0 &  u_0+u_1 & u_0   \\
\end{bmatrix}.  $$

Conversely, if we start with $\mathbf{u}$ and define $d_{\mathbf{w}}(m)$ by \eqref{eq:d_w}, then the sequence of weights $\mathbf{w}$ can be recovered using the relations
\begin{align*}
w_0 &= u_0, \\
w_n &= \sum_{j=0}^n 2^{n-j} u_j, \quad n \geq 1.
\end{align*} 
As a first auxiliary result, we obtain a recurrence relation for the matrices $G_{\mathbf{w}}(n)$. 

\begin{prop} \label{prop:G_recursion}
For all $k \in \N, n \in \N_+$ we have
\begin{align*}
    G_{\mathbf{w}}(2^k n) = J_n \otimes G_{\mathbf{w}}(2^k) + G_{\mathbf{w}'}(n) \otimes L_{2^k}, 
\end{align*}
    where $\mathbf{w}'=(w_j')_{j \geq 0}$ depends only on $k$ and is defined by
    $$  w'_j = w_{j+k} - 2^j w_k.$$
\end{prop}
\begin{proof}
Write $G_{\mathbf{w}}(2^k n)$ as an $n \times n$ block matrix:
$$ G_{\mathbf{w}}(2^k n) = \begin{bmatrix}
    G_0 & G_1 &  \cdots & G_{n-1} \\
    G_1 & G_2 &  \cdots & G_{n} \\
    \vdots & \vdots & & \vdots \\
    G_{n-1} & G_{n} &\cdots & G_{2n-2} 
\end{bmatrix}, $$
where the blocks $G_l$ have size $2^k \times 2^k$. Consider the entry at position $(i,j)$ in $G_l$, that is, $d_{\mathbf{w}}(2^k l + i+j)$. If $i + j \neq 2^k-1$, then $\nu_2(2^kl+i+j+1)= \nu_2(i+j+1)$, so by \eqref{eq:d_w}  we get $d_{\mathbf{w}}(2^k l + i+j) =  d_{\mathbf{w}}(i+j)$. 
Otherwise, if $i + j  = 2^k-1$, then
$$ d_{\mathbf{w}}(2^k l + i+j) =\sum_{m=0}^{k + \nu_2(l+1)} u_m = d_{\mathbf{w}}(2^k-1) + \sum_{m=1}^{\nu_2(l+1)} u_{k+m}  = d_{\mathbf{w}}(i+j) + d_{\mathbf{w}'}(l),  $$ 
since $\mathbf{w}'$ corresponds to $\mathbf{u}' = (0,u_{k+1}, u_{k+2}, \ldots)$.

This means that
$$   G_l =  G_{\mathbf{w}}(2^k) +  d_{\mathbf{w}'}(l) L_{2^k}, $$
and the result follows.    
\end{proof}

We now apply this identity to obtain a recursion for $\mathcal{G}_{\mathbf{w}}(n)$ and $\widetilde{\mathcal{G}}_{\mathbf{w}}(n)$. This is a generalization of \cite[Lemmas 8, 9]{FKS}. 

\begin{prop} \label{prop:det_G_recurrence}
Let  $n \in \N_{\geq 2}$ and $k \in \N$ be such that $2^k < n \leq 2^{k+1}$.
The following recursive formulas hold.
\begin{enumerate}[label=(\alph*)]
    \item If $2^k < n \leq 3 \cdot 2^{k-1}$, then
    \begin{align*}
        \mathcal{G}_{\mathbf{w}}(n) &= (-1)^n (w_{k+1}-2w_k)^{2n-2^{k+1}} \mathcal{G}_{\mathbf{w}}(2^{k+1}-n), \\
        \widetilde{\mathcal{G}}_{\mathbf{w}}(n) &= (-1)^n (w_{k+1}-2w_k)^{2n-2^{k+1}-1} \widetilde{\mathcal{G}}_{\mathbf{w}}(2^{k+1}-n+1).
    \end{align*}  
    \item If $3 \cdot 2^{k-1} < n \leq 2^{k+1}$, then
    \begin{align*}
        \mathcal{G}_{\mathbf{w}}(n) &= \varepsilon(n) 2^{2n-3 \cdot 2^k}(w_{k+1}-2w_k)^{2^k} \mathcal{G}_{\mathbf{w}'}(n-2^k), \\
        \widetilde{\mathcal{G}}_{\mathbf{w}}(n) &= \varepsilon(n) 2^{2n-3 \cdot 2^k-1}(w_{k+1}-2w_k)^{2^k} \widetilde{\mathcal{G}}_{\mathbf{w}'}(n-2^k),
    \end{align*}
     where $\mathbf{w}' = (w_0, w_1, \ldots, w_{k-1}, \frac{1}{2} w_{k+1})$ and
$$ \varepsilon(n) = \begin{cases}                          
    1 &\text{if } n \neq 2,4, \\
    -1 &\text{if } n=2,4.
\end{cases}$$
\end{enumerate}
\end{prop}
\begin{proof}
We carry out the proof in detail for $\mathcal{G}_\mathbf{w}(n)$ and later describe the (small) changes that occur in the case of  $\widetilde{\mathcal{G}}_\mathbf{w}(n)$.
Proposition \ref{prop:G_recursion} applied to $n=2$ yields
$$ G_{\mathbf{w}}(2^{k+1}) = 
\begin{bmatrix}
    1 & 1 \\ 1 & 1
\end{bmatrix} \otimes G_{\mathbf{w}}(2^k) + 
\begin{bmatrix}
    0 & u_{k+1} \\ u_{k+1} & 0
\end{bmatrix} \otimes L_{2^k}. $$
 By subtracting the top row of blocks from the bottom one, and then subtracting the left column from the right one, we get the following block matrix:
$$
\begin{bmatrix}
    1&0\\0&0
\end{bmatrix} \otimes G_{\mathbf{w}}(2^k)
+ u_{k+1}\begin{bmatrix}
    0&1\\1&-2
\end{bmatrix} \otimes L_{2^k}.
$$
Its $n \times n$ (non-block) submatrix lying in the upper left corner can be obtained from $G_{\mathbf{w}}(n)$ by corresponding elementary operations.

In case (a), this submatrix can be written in block form
\begin{equation} \label{eq:G_a}
\begin{NiceArray}{cccl}[right-margin=.4em]
    \Block[borders={bottom,right,tikz=dashed}]{2-2}<\Large>{G_{\mathbf{w}}(2^k)} &  & \mathbf{0} & 2^{k+1}-n \\  &  & u_{k+1} L &n-2^k \\ \mathbf{0} & u_{k+1} L & \mathbf{0} &n-2^k
    \CodeAfter
\SubMatrix[{1-1}{3-3}]
\end{NiceArray},
\end{equation}
where we follow the convention of writing block heights to the right of a symmetric block matrix. 
Applying the generalized Laplace expansion with respect to the bottom row and then the rightmost column of blocks, we get
$$ \mathcal{G}_\mathbf{w}(n) = (-1)^{(n-2^k)(2^k+1)} u_{k+1}^{2n-2^{k+1}} \mathcal{G}_\mathbf{w}(2^{k+1}-n)  (\det L_{n-2^k})^2,  $$
and the formula in (a) follows quickly.

In case (b), for $n=2,4$, the desired formula can be verified directly. Otherwise, we can assume $k \geq 2$. The considered $n \times n$ submatrix contains nonzero elements in the bottom right block and can be written in the form
\begin{equation} \label{eq:G_b}
\begin{NiceArray}{cccccl}[right-margin=.4em]
    \Block[borders={bottom,right,tikz=dashed}]{3-3}<\Large>{G_{\mathbf{w}}(2^k)} &  &  & \mathbf{0} & \mathbf{0} & 2^{k+1}-n \\ &  &  & \mathbf{0} & u_{k+1}L & 2n-3 \cdot 2^k \\  & & & u_{k+1}L & \mathbf{0} & 2^{k+1}-n \\ \mathbf{0} & \mathbf{0} & u_{k+1}L  & \mathbf{0} & \mathbf{0} & 2^{k+1}-n \\ \mathbf{0} & u_{k+1}L & \mathbf{0} & \mathbf{0} & -2u_{k+1}L & 2n-3 \cdot 2^k
   \CodeAfter
\SubMatrix[{1-1}{5-5}]
\end{NiceArray}.
\end{equation}
Expanding with respect to the penultimate row and then column of blocks, we get
$$
    \mathcal{G}_{\mathbf{w}}(n) 
    =  (-1)^n u_{k+1}^{2^{k+2}-2n} \det 
\begin{NiceArray}{ccccl}[left-margin=.4em,right-margin=.4em]
    \Block[borders={bottom,right,tikz=dashed}]{2-3}{G_{\mathbf{w}}(n-2^k)}  & & &  \mathbf{0} & 2^{k+1}-n \\  & & & u_{k+1}L & 2n-3 \cdot 2^k \\  \mathbf{0} &  & u_{k+1}L & -2u_{k+1}L & 2n-3 \cdot 2^k
   \CodeAfter
\SubMatrix[{1-1}{3-4}]
\end{NiceArray}.  
$$
Now, we add the bottom row of blocks multiplied by $1/2$ to the middle row. But then $\frac{1}{2} u_{k+1}$ is added precisely at positions $(i,j)$ in $G_{\mathbf{w}}(n-2^k)$ such that $i+j=2^k-1$, that is, positions where $u_k$ occurs. This operation thus replaces $u_k$ with $u_k' = u_k + \frac{1}{2}u_{k+1}$, or equivalently, replaces $w_k$ with $w_k' = \frac{1}{2} w_{k+1}$. We thus get
$$ \mathcal{G}_\mathbf{w}(n) = (-1)^n u_{k+1}^{2^{k+2}-2n} \mathcal{G}_\mathbf{w'}(n-2^k) \det (-2u_{k+1}L),$$
and (b) follows.

The formulas for $\widetilde{\mathcal{G}}_{\mathbf{w}}(n)$ are proved very similarly, by analogous operations on rows and columns. The only significant difference is that we leave the rightmost (non-block) column, consisting of $1$s, intact when performing operations on columns. As a result, we obtain a matrix of the form \eqref{eq:G_a} or \eqref{eq:G_b} in the case (a) or (b), respectively, where the top $2^k$ entries in the last column are replaced with $1$s, and the bottom $n-2^k$ entries with $0$s. The remaining steps are essentially the same, except when applying  generalized Laplace expansion with respect to columns, the rightmost column is excluded.
\end{proof}

\subsection{Explicit formulas}

Through repeated application of Proposition \ref{prop:det_G_recurrence}, we now obtain a closed form expression for the determinants $\mathcal{G}_{\mathbf{w}}(n)$. First, we consider the special case $n=2^k$, where said formula takes a particularly simple form.

\begin{thm} \label{thm:det_G_2^k}
    For all $k \in \N$ we have
    $$  \mathcal{G}_{\mathbf{w}}(2^{k})  = \gamma(k) \, w_k \prod_{j=1}^{k} (w_k - 2^{j} w_{k-j})^{2^{k-j}},$$
    where
    $$\gamma(k) = \begin{cases}
    1 &\text{if } k \neq 1, \\
    -1 &\text{if } k =1.
    \end{cases}$$
\end{thm}
\begin{proof}
By direct computation, the formula holds for $k=0,1,2$. Now, let $k \geq 2$ and assume that the statement holds for $k$.  By Proposition \ref{prop:det_G_recurrence}(b) we have
$$  \mathcal{G}_{\mathbf{w}}(2^{k+1}) =  2^{2^k}(w_{k+1}-2w_k)^{2^k} \mathcal{G}_{\mathbf{w}'}(2^k), $$
where $\mathbf{w}' = (w_0, w_1, \ldots, w_{k-1}, \frac{1}{2} w_{k+1})$. By the inductive assumption, we get
\begin{align*}  
\mathcal{G}_{\mathbf{w}}(2^{k+1}) &=  2^{2^k}(w_{k+1}-2w_k)^{2^k} \cdot \frac{1}{2} w_{k+1} \prod_{j=1}^{k} \left(\frac{1}{2} w_{k+1} - 2^{j} w_{k-j}\right)^{2^{k-j}}  \\
&= 2^m w_{k+1} \prod_{j=1}^{k+1} \left( w_{k+1} - 2^{j} w_{k+1-j}\right)^{2^{k+1-j}},
\end{align*} 
where
$$m= 2^k -1 - \sum_{j=1}^k 2^{k-j} = 0.$$
The result follows.
\end{proof}

We move on to state the main result of this section, namely a general formula for $\mathcal{G}_{\mathbf{w}}(n)$ in terms of the binary expansion of $n$. Here, for a binary digit $a \in \{0,1\}$ by $a^l$ we mean the string consisting of $a$ repeated $l$ times.

\begin{thm} \label{thm:det_G}
Write the binary expansion of $n \in \N_+$ in the following way:
$$(n)_2= 1^{l_s} 0^{l_{s-1}} 1^{l_{s-2}} \cdots  \overline{a}^{l_3} a^{l_2} 1 0^{l_1}, $$
where $s \geq 1$, $l_1 \geq 0$, $l_i \geq 1$ for $i=2,\ldots,s$ and $a \in \{0,1\}$, $\overline{a}=1-a$.

Then
\begin{align*}
    \mathcal{G}_{\mathbf{w}}(n) = &  (-1)^\frac{n(n-1)}{2} w_{l_1} \prod_{i=1}^s \prod_{j=1}^{l_i}(w_{k_i}-2^jw_{k_i-j})^{\alpha_{i,j}},
\end{align*} 
where we set
\begin{align*}
k_i &= \begin{cases}
l_1 &\text{if } i=1, \\
    1+\sum_{j=1}^i l_i &\text{if } i \geq 2,
\end{cases} \\
\alpha_{i,j} &=  2^{k_i-j} \qquad \text{if } 1 \leq j < l_i, \\
\alpha_{i,l_i}&= \begin{cases}
    1 &\text{if } i=1, \\ 
    2^{l_1+1} &\text{if } i=2,3, \\ 
    2^{k_{i-2}+1} - \alpha_{i-1,l_{i-1}} &\text{if } i \geq 4.
    \end{cases}
\end{align*}
In particular, the exponents $\alpha_{i,j}$ are all positive, strictly decreasing with respect to $j=1,\ldots,l_i$, and satisfy
$$ 1+ \sum_{i=1}^s \sum_{j=1}^{l_i} \alpha_{i,j} = n.$$
\end{thm}
\begin{proof}
We perform induction on the length $\ell(n)$ of the prefix
$$  1^{l_s} 0^{l_{s-1}} 1^{l_{s-2}} \cdots  \overline{a}^{l_3} a^{l_2} $$
of the binary expansion of $n$. 

If $\ell(n) = 0$, then $n = 2^{l_1}$ and $s=1$. The formula follows from Theorem \ref{thm:det_G_2^k} with $k = l_1$.
We also separately consider the case $\ell(n)=1$, namely $n = 3 \cdot 2^{l_1}$ and $s=2,l_2=1$. Part (a) of Proposition \ref{prop:det_G_recurrence} applied to $k = l_1+1$ yields
$$ \mathcal{G}_{\mathbf{w}}(n) = (-1)^n (w_{k_2}-2w_{k_2-1})^{2^{l_1+1}} \mathcal{G}_{\mathbf{w}}(2^{l_1}),$$
which leads to the desired formula.

Before proceeding to the induction step, observe that for $1 \leq i \leq s-2$ the numbers 
$$n_{i} := \frac{1}{2}\alpha_{i+2,l_{i+2}}$$
have binary expansion
\begin{equation*} 
    (n_i)_2 = 1^{l_i} 0^{l_{i-1}} \cdots  \overline{b}^{l_3} b^{l_2} 1 0^{l_1},
\end{equation*}  
where $b  = (i+1) \bmod{2}$. We may also define $n_{s-1}, n_s$ in the same way, where in particular $n_s = n$. Moreover, we have the inequalities $2^{k_i-1} < n_i \leq 2^{k_i}$ and recurrence relations
$$ n_i = 2^{k_i}-n_{i-1}.$$

Now, let $\ell(n) \geq 2$ (implying $s \geq 2)$ and suppose that our claim holds for all $m \in \N$ with $\ell(m) < \ell(n)$. 
We consider two cases, depending on the second most significant digit in the binary expansion of $n$. If this digit is $0$ (so $s \geq 3$ and $l_s=1$), then $2^{k_s-1}< n < 3 \cdot 2^{k_s-2}$. We apply part (a) of Proposition \ref{prop:det_G_recurrence}, where $k = k_s-1$, to obtain
\begin{equation} \label{eq:det_G_1}
   \mathcal{G}_{\mathbf{w}}(n) = (-1)^n (w_{k_s}-2w_{k_s-1})^{2n-2^{k_s}} \mathcal{G}_{\mathbf{w}}(2^{k_s}-n).
\end{equation} 
The exponent is 
$$2n-2^{k_s} = 2(2^{k_s}-n_{s-1})-2^{k_s}= 2(2^{k_{s-1}} -n_{s-1}) = 2n_{s-2} =\alpha_{s,l_s}.$$
where we used $k_s = k_{s-1}+1$.
Furthermore, $2^{k_s}-n = n_{s-1}$, and by the inductive assumption the formula for $ \mathcal{G}_{\mathbf{w}}(n_{s-1})$ is exactly as in the statement, except the upper limit in the first product is $s-1$. Inserting it into \eqref{eq:det_G_1}, we get the desired formula (up to sign). 
Since $n_{s-1} \equiv -n \pmod{4}$, the sign equals $-1$ raised to the power
$$  n + \frac{-n(-n-1)}{2} \equiv \frac{n(n-1)}{2} \pmod{2},$$
and the formula is proved.

Now, assume that the second most significant digit in the binary expansion of $n$ is $1$. 
This means that  $3 \cdot 2^{k_s-2} < n < 2^{k_s}$, after excluding the already considered case $n=3\cdot 2^{l_1}$. We apply Proposition \ref{prop:det_G_recurrence}(b) to $k = k_s-1$ and get
\begin{equation} \label{eq:det_G_2}
    \mathcal{G}_{\mathbf{w}}(n) = 
    2^{2n-3 \cdot 2^{k_s-1}}(w_{k_s}-2w_{k_{s}-1})^{2^{k_s-1}} \mathcal{G}_{\mathbf{w}'}(n-2^{k_s-1}),
\end{equation} 
where
$$\mathbf{w}' = (w_0', w_1', \ldots, w_{k_{s}-2}', w_{k_{s}-1}') = (w_0, w_1, \ldots, w_{k_{s}-2}, \frac{1}{2}w_{k_s}).$$
The binary expansion of $n'=n-2^{k_s-1}$ is obtained by removing the leading $1$ from the expansion of $n$. Hence, the values $l'_i, k'_i, \alpha'_{i,j}$, corresponding to $n'$ agree with $l_i,k_i,\alpha_{i,j}$ for all $i \leq s-1$. For $i=s$ we have $l_s'=l_s-1,k_s'=k_s-1$ and 
\begin{align*}
    \alpha'_{s,j} &= 2^{l_s-1-j} = \alpha_{s,j+1} \qquad \text{if } j<l_s-1;\\
    \alpha'_{s,l_{s}-1} &= \alpha'_{s,l'_s} = \begin{cases}
    2^{l_1+1} &\text{if } s=2,3, \\ 
    2^{k_{s-2}+1} - \alpha_{s-1,l_{s-1}} &\text{if } s \geq 4
\end{cases} = \alpha_{s,l_s}
\end{align*} 
By the inductive assumption and $n' \equiv n \pmod{4}$, we thus get
\begin{align*}    
\mathcal{G}_{\mathbf{w}'}(n') = &(-1)^{\frac {n(n-1)}{2}} w_{l_1} \prod_{i=1}^{s-1} \prod_{j=1}^{l_i}(w_{k_i}-2^jw_{k_i-j})^{\alpha_{i,j}} \\ 
&\left( \prod_{j=1}^{l_s-2}\left(\frac{1}{2}w_{k_s}-2^jw_{k_s-1-j}\right)^{\alpha_{s,j+1}}\right) \left(\frac{1}{2}w_{k_s}-2^{l_s-1}w_{k_s-l_s}\right)^{\alpha_{s,l_s}}.
\end{align*}
The first line already appears as a factor in the desired formula. Hence, we focus on the second line, which after some manipulation becomes
$$ 2^{-m}\prod_{j=2}^{l_s} \left( w_{k_s}-2^jw_{k_s-j}\right)^{\alpha_{s,j}} , $$
where
$$ m = \sum_{j=2}^{l_s} \alpha_{s,j} = \sum_{j=2}^{l_s-1} 2^{k_s-j} + \alpha_{s,l_s} =  2^{k_s-1} -2^{k_s-l_s+1} + \alpha_{s,l_s}. $$
If $s=2$, we have $n = 2^{k_2}-2^{k_1}$ and $k_2-l_2=k_1+1$, so
$$ m = 2^{k_2-1}-2^{k_1+2}+2^{k_1+1} =  2^{k_2-1}-2^{k_1+1} = 2n -3 \cdot 2^{k_2-1}.$$
If $s \geq 3$, we also get 
$$ m =  2^{k_s-1} - 2^{k_{s-1}+1} +2n_{s-2} = 2^{k_s-1} -2n_{s-1} = 2n - 3 \cdot 2^{k_s - 1}.$$
Taking all into account, after plugging the formula for $\mathcal{G}_{\mathbf{w}'}(n')$ into \eqref{eq:det_G_2} the exponents of $2$ cancel out, and we get the desired identity.

Finally, positivity and monotonicity of the exponents $\alpha_{i,j}$ with respect to $j$ follows directly from their definition. The fact that they sum up to $n-1$ can be verified similarly as in the calculation of $m$ above. However, a simpler approach is to observe that $\mathcal{G}_{\mathbf{w}}(n)$ is either a homogeneous degree $n$ polynomial in $w_0,w_1,\ldots$ (which would imply our claim), or identically equal to $0$. But the latter cannot occur, by the formula we have just proved.
\end{proof}

\begin{exam}
Let us compute $\cal{G}_{\bf w}(54)$ using Theorem \ref{thm:det_G}. We have $(54)_{2}=110110$, i.e., $s=4$ and
$$
(l_{1},k_{1})=(1,1),\quad (l_{2},k_{2})=(1,3),\quad (l_{3},k_{3})=(1,4),\quad (l_{4},k_{4})=(2,6).
$$
The corresponding values of $\alpha_{i,j}$ are the following
$$
\alpha_{1,1}=1,\quad \alpha_{2,1}=4,\quad \alpha_{3,1}=4,\quad \alpha_{4,1}=32,\quad \alpha_{4,2}=12.
$$
Consequently we get
$$
\cal{G}_{\bf w}(54)=w_{1}(w_{1}-2w_{0})(w_{3}-2w_{2})^{4}(w_{4}-2w_{3})^4(w_{6}-2w_{5})^{32}(w_{6}-4w_{4})^{12}.
$$
\end{exam}
\begin{rem}
{\rm Theorem \ref{thm:det_G} shows that the structure of $\mathcal{G}_{\mathbf{w}}(n)$ can be seen from the pattern of consecutive blocks in the binary expansion of $n$. More precisely, the relevant indices $k_i$ and exponents $\alpha_{i,j}$ are determined by the lengths of these blocks, so the factorization reflects the 2-adic structure of $n$ rather than its size alone.}
\end{rem}

It should be possible to obtain a similar prodduct formula for the determinants $\widetilde{\mathcal{G}}_{\mathbf{w}}(n)$. However, this seems more tricky, due to the additional ``$+1$'' in the index in Proposition \ref{prop:det_G_recurrence}(a), which makes the indices harder to control when applying the proposition repeatedly. Instead, we settle for the following qualitative description, which will be sufficient for our purposes.

\begin{prop} \label{prop:det_Gt}
    Let $n \in \N_{\geq 2}$ and let $k \in \N$ be such that $2^k < n \leq 2^{k+1}$. Then we have the following:
    \begin{enumerate}[label=(\alph*)]
        \item $\widetilde{\mathcal{G}}_{\mathbf{w}}(n)$ is, up to sign, a product of $n-1$ factors of the form $(w_l - 2^{l-j} w_j)$, where $0 \leq j < l \leq k+1$;
        \item the total multiplicity of factors with $l=k+1$ is equal to $2n-2^{k+1}-1$.
    \end{enumerate}
\end{prop}
\begin{proof}
We prove (a) and (b) simultaneously by induction on $n$. We have $\widetilde{\mathcal{G}}_{\mathbf{w}}(1)=1$ so both statements hold for $n = 1$. Now, let $n \geq 2$ and consider two cases. If $2^k < n \leq 3 \cdot 2^{k-1}$, then our claim follows directly from Proposition \ref{prop:det_G_recurrence}(a).
If $ 3 \cdot 2^{k-1} <n \leq 2^{k+1} $, then by Proposition \ref{prop:det_G_recurrence}(b) we have
$$ \widetilde{\mathcal{G}}_{\mathbf{w}}(n)  = 2^{2n-3\cdot 2^k-1} (w_{k+1}-2w_k)^{2^k} \widetilde{\mathcal{G}}_{\mathbf{w'}}(n-2^k),$$
where $\mathbf{w}' = (w_0, w_1, \ldots, w_{k-1}, \frac{1}{2} w_{k+1})$. The inductive assumption says that the total multiplicity of factors of the form $(\frac{1}{2}w_{k+1}-2^{k-j}w_j)$ in $\widetilde{\mathcal{G}}_{\mathbf{w'}}(n-2^k)$ is 
$$ 2(n-2^k) - 2^k -1 = 2n - 3 \cdot 2^k-1. $$
Distributing the power $2^{2n-3\cdot 2^k-1}$ over these factors and taking $(w_{k+1}-2w_k)^{2^k}$ into account, we again obtain our claim.
\end{proof}

\subsection{Special cases}

We now give some applications of Theorem \ref{thm:det_G}. The first one relies on a simple but somewhat unexpected connection, which we state as a proposition. More precisely, the period-doubling sequence $\mathbf{d}=(d_n)_{n \geq 0}$ and its bit-wise negation  $\bar{\mathbf{d}}=(\bar{d}_n)_{n \geq 0}$ turn out to be instances of $\mathbf{d}_\mathbf{w}$ where the weights are (shifted) Jacobsthal numbers. We follow the definition from \cite{FKS}, where 
$$d_n := (\nu_2(n+1)+1) \bmod{2}.$$ 
We note that the variant with $0$s and $1$s swapped ($\mathbf{d}$ becomes $\bar{\mathbf{d}}$ and vice versa) seems to be more common. Also, recall that Jacobsthal numbers are defined by $J(0)=0,J(1)=1$ and $J(n+2)=J(n+1)+2J(n)$ for $n \in \N$.

\begin{prop}
We have the following:
\begin{enumerate}[label=(\alph*)]
    \item If $\mathbf{w} = (J(j+1))_{j \geq 0}$, then $\mathbf{d}_{\mathbf{w}}$ is the period-doubling sequence $\mathbf{d}$;
    \item If $\mathbf{w} = (J(j))_{j \geq 0}$, then $\mathbf{d}_{\mathbf{w}}$ is the bit-wise negation $\bar{\mathbf{d}}$ of the period-doubling sequence.
\end{enumerate}
\end{prop}
\begin{proof}
    In the case (a) we have $u_0=J_1=1$ and $u_j=J(j+1)-2J(j) = (-1)^{j}$ for $j \geq 1$, which is a well-known identity. It follows that
    $$d_{\mathbf{w}}(n) = \sum_{j=0}^{\nu_2(n+1)} (-1)^{j} = (\nu_2(n+1)+1) \bmod{2} = d_n.$$
    In the case (b), we have $u_0 = J_0 = 0$ and $u_j = J(j) - 2J(j-1) = (-1)^{j-1}$ for $j \geq 1$. A similar computation leads to $d_{\mathbf{w}}(n) = \bar{d}_n$.
\end{proof}

As a corollary, we obtain a slight generalization of the result by Fokkink, Kraaikamp and Shallit \cite[Theorem 2]{FKS} (and more general Theorem 10 from the same paper) concerning Hankel determinants for the period-doubling sequence.
 There appears to be a small mistake in the original statement of the latter theorem, namely the conditions modulo $4$ governing the sign of the determinant are shifted by $1$.

\begin{prop} \label{prop:G_period_doubling}
Let $n \in \N_+$ and, with the notation of Theorem \ref{thm:det_G}, put
$$ l = \max_{1 \leq i \leq s} l_i, \qquad A_j = \sum_{i:\, l_i \geq j}^{s} \alpha_{i,j}. $$
Then we have the following:
\begin{enumerate}[label=(\alph*)]
    \item The $n$th Hankel determinant for the period-doubling sequence $\mathbf{d}$ is 
        $$  (-1)^{\frac{(n-1)(n-2)}{2}} J(l_1+1) \prod_{j=1}^{l} J(j)^{A_j}.   $$
    \item The $n$th Hankel determinant for the bit-wise negation $\bar{\mathbf{d}}$ of the period-doubling sequence is 
    $$ (-1)^{\frac{n(n-1)}{2}} J(l_1) \prod_{j=1}^{l} J(j)^{A_j}.   $$
    In particular, the determinant is nonzero if and only if $n$ is even.
\end{enumerate}
\end{prop}
\begin{proof}
 In case (a), put $\mathbf{w} = (J(j+1))_{j \geq 0}$. The product over $i,j$ in Theorem \ref{thm:det_G} consists of factors of the form
    $$ w_{k_i} - 2^j w_{k_i-j} = J(k_i+1) - 2^j J(k_i+1-j) = (-1)^{k_i+1-j} J(j). $$
   After rearranging, we get the product as in the statement (up to sign).
  The sign equals $-1$ raised to the power
    $$
        \frac{n(n-1)}{2} + \sum_{i=1}^s \sum_{j=1}^{l_i} (k_i+1 - j) \alpha_{i,j}.$$ 
   The only odd $\alpha_{i,j}$ is $\alpha_{1,l_1}=1$, which only occurs when $l_1 \geq 2$, i.e., $n$ is even. The corresponding summand equals $1$, and we obtain the sign as in the statement.
   
    The proof of (b) is similar. This time, the factors are
$$ w_{k_i} - 2^j w_{k_i-j} = J(k_i) - 2^j J(k_i-j) = (-1)^{k_i-j} J(j), $$
and the sign equals $-1$ raised to the power
    $$
        \frac{n(n-1)}{2} + \sum_{i=1}^s \sum_{j=1}^{l_i} (k_i - j) \alpha_{i,j}.$$
  All summands are even, and we again obtain the desired formula.
\end{proof}

We now apply Theorem \ref{thm:det_G} to the weights $w_j = t^j$, where
we adopt the notation
$$ \mathcal{G}(n,t) := \mathcal{G}_{\mathbf{w}}(n) \qquad \text{for } \mathbf{w}=(t^j)_{j \geq 0}. $$
The resulting characterization of the roots $\mathcal{G}(n,t)$  will play a  role  in Section \ref{sec:vanishing}. 

\begin{cor} \label{cor:det_G(n,t)}
Let $n \in \N_+$ and, with the notation of Theorem \ref{thm:det_G}, put
$$ A=l_1 + \sum_{i=1}^s \sum_{j=1}^{l_1} (k_i-j) \alpha_{i,j}, \qquad l = \max_{1 \leq i \leq s} l_i, \qquad A_j = \sum_{i:\, l_i \geq j}^{s} \alpha_{i,j}. $$
Then 
$$ \mathcal{G}(n,t) = (-1)^{\frac{n(n-1)}{2}} t^A \prod_{j=1}^{l} (t^j - 2^j)^{A_j},$$
 In particular, the roots of $\mathcal{G}(n,t)$ are precisely $t=0$ and $t=2\zeta$, where $\zeta$ is any root of unity of order $\leq l$.
\end{cor}

\section{Formulas for the determinants $\mathcal{H}_{\mathbf{w}}(n)$}\label{sec:H}

The main goal of this section is to derive formulas for the determinants $\mathcal{H}_{\mathbf{w}}(n)$. In the general case we give a set of recursive formulas, involving a sequence of determinants of certain auxiliary matrices. For special choices of weights $\mathbf{w}$ and indices $n$ we provide explicit expressions, including \eqref{eq:simple_expression} and \eqref{eq:polynomial_simple_expression}. The overall approach is similar as in the previous section, although some additional complications arise.

\subsection{Recurrence relations}
  We first give a recurrence relation satisfied by the matrices $H_{\mathbf{w}}(n)$. In its formulation, we use the notation 
$$F_{\mathbf{w}}(n) := G_{\mathbf{w}}(n) - w_0 J_n = [d_{\mathbf{w}}(i+j)-w_0]_{0 \leq i,j < n}. $$

 \begin{prop} \label{prop:H_recursion}
    For all $k,n \in \N$ we have
    $$H_{\mathbf{w}}(2^k n) = J_n \otimes H_{\mathbf{w}}(2^k) + H_{\mathbf{w}^{(k)}}(n) \otimes J_{2^k}  + F_{\mathbf{w}^{(k)}}(n) \otimes K_{2^k},$$
    where $\mathbf{w}^{(k)} = (w_k,w_{k+1},\ldots)$ denotes the $k$-fold shift of $\mathbf{w}$.
\end{prop}
\begin{proof}
The proof is similar to Proposition \ref{prop:G_recursion}.
We write the matrix $H_{\mathbf{w}}(2^k n)$ in block form
$$ H_{\mathbf{w}}(2^k n) = \begin{bmatrix}
    H_0 & H_1 &  \cdots & H_{n-1} \\
    H_1 & H_2 &  \cdots & H_{n} \\
    \vdots & \vdots & & \vdots \\
    H_{n-1} & H_{n} &\cdots & H_{2n-2} 
\end{bmatrix}, $$
where the blocks $H_l$ have size $2^k \times 2^k$. The entry at position $(i,j)$ in $H_l$ is
$$ s_{\mathbf{w}}(2^k l + i+j) = \begin{cases}
    s_{\mathbf{w}^{(k)}}(l) + s_{\mathbf{w}}(i+j) &\text{if } i+j < 2^m, \\
    s_{\mathbf{w}^{(k)}}(l+1) + s_{\mathbf{w}}(i+j)-w_k &\text{if } i+j \geq 2^m.
\end{cases}  $$
This means that
$$   H_l = H_{\mathbf{w}}(2^k) + s_{\mathbf{w}^{(k)}}(l) J_{2^k} +(s_{\mathbf{w}^{(k)}}(l+1)-s_{\mathbf{w}^{(k)}}(l)-w_k) K_{2^k}, $$
and the result follows.
\end{proof}

We now move on to prove a general recursive formula for determinants $\mathcal{H}_{\mathbf{w}}(n)$. We state it in vector form, which contains the recurrence for $\widetilde{\mathcal{G}}_{\mathbf{w}}(n)$ from Proposition \ref{prop:det_G_recurrence}. 

\begin{thm} \label{thm:det_H_recurrence}
Let  $n \in \N_{\geq 5}$ and $k \in \N$ be such that $2^k < n \leq 2^{k+1}$. 
The following recursive formulas hold.
\begin{enumerate}[label=(\alph*)]
    \item If $2^k < n \leq 3 \cdot 2^{k-1}$, then
$$\begin{bmatrix}
    \mathcal{H}_{\mathbf{w}}(n) \\
    \widetilde{\mathcal{G}}_{\mathbf{w}}(n)
\end{bmatrix} = 
A_{\mathbf{w}}(n) \begin{bmatrix}
    -1 & C_{\mathbf{w}}(n) \\ 0 & 1
\end{bmatrix}
\begin{bmatrix}
    \mathcal{H}_{\mathbf{w}}(2^{k+1}-n+1) \\
    \widetilde{\mathcal{G}}_{\mathbf{w}}(2^{k+1}-n+1)
\end{bmatrix},
$$
where 
\begin{align*}
    A_{\mathbf{w}}(n) &=  (-1)^n (w_{k+1}-2w_k)^{2n-2^{k+1}-1}, \\
    C_{\mathbf{w}}(n) & = (-1)^n \frac{w_k^2}{w_{k+1}-2w_k}.
\end{align*}
    \item If $3 \cdot 2^{k-1} < n \leq 2^{k+1}$, then
    $$\begin{bmatrix}
    \mathcal{H}_{\mathbf{w}}(n) \\
    \widetilde{\mathcal{G}}_{\mathbf{w}}(n)
\end{bmatrix} = 
B_{\mathbf{w}}(n) \begin{bmatrix}
    1 & C_{\mathbf{w}}(n) \\ 0 & 1
\end{bmatrix}
\begin{bmatrix}
    \mathcal{H}_{\mathbf{w}'}(n-2^k) \\
    \widetilde{\mathcal{G}}_{\mathbf{w}'}(n-2^k)
\end{bmatrix},
$$
where $C_{\mathbf{w}}(n)$ is as in (a) and
\begin{align*}
   B_{\mathbf{w}}(n) &= 2^{2n-3\cdot 2^k-1} (w_{k+1}-2w_k)^{2^k}, \\ 
   \mathbf{w}' &= (w_0, w_1, \ldots, w_{k-1}, \frac{1}{2} w_{k+1})
\end{align*}
 \end{enumerate}
\end{thm}
\begin{proof}
We only need to prove the formula for $\mathcal{H}_{\mathbf{w}}(n)$. First, we apply Proposition \ref{prop:H_recursion} to $n=2$, obtaining
$$  H_{\mathbf{w}}(2^{k+1}) = 
\begin{bmatrix}
    1&1\\1&1
\end{bmatrix} \otimes H_{\mathbf{w}}(2^{k}) + 
\begin{bmatrix}
    0&w_k\\w_k&w_{k+1}
\end{bmatrix} \otimes J_{2^{k}}+ 
u_{k+1}\begin{bmatrix}
    0&1\\1&0
\end{bmatrix} \otimes K_{2^k},$$
where $u_{k+1} = w_{k+1}-2w_k$. We now perform elementary row and column operations common for both cases (a) and (b). 
First, we subtract the top row of blocks from the bottom one, and then subtract the left column from the right, obtaining 
\begin{equation} \label{eq:H_after_reduction}
\begin{bmatrix}
    1&0\\0&0
\end{bmatrix} \otimes H_{\mathbf{w}}(2^{k}) + 
\begin{bmatrix}
    0&w_k\\w_k&u_{k+1}
\end{bmatrix} \otimes J_{2^k} + 
u_{k+1}\begin{bmatrix}
    0&1\\1&-2
\end{bmatrix} \otimes K_{2^k}.    
\end{equation}
Next, we subtract the $2^k$th (non-block) row from all rows below it, and the $2^k$th column from all columns to the right. This only affects the part involving $J_{2^k}$, and the whole matrix becomes
\[\begin{NiceArray}{ccccl}[right-margin=.4em]
\Block[borders={bottom,right,tikz=dashed}]{2-2}<\Large>{H_{\mathbf{w}}(2^k)} &    & w_k & \mathbf{0} &  1 \\
 &   & w_k \mathbf{1} & u_{k+1} M & 2^k-1 \\[-1em]  \\ %tiny extra row to avoid overlap
w_k  & w_k \mathbf{1}^T  & u_{k+1} & \mathbf{0} &   1 \\
\mathbf{0} & u_{k+1} M &  \mathbf{0}  & -2u_{k+1}M &  2^k-1 
\CodeAfter
\SubMatrix[{1-1}{5-4}]
\end{NiceArray}.\]

From now on, we restrict our attention to the $n\times n$ submatrix lying in the upper left corner. In case (a) this submatrix is
\[\begin{NiceArray}{ccccl}[right-margin=.4em]
\Block[borders={bottom,right,tikz=dashed}]{2-2}<\Large>{H_{\mathbf{w}}(2^k)} &    & w_k  \mathbf{1} & \mathbf{0} &  2^{k+1}-n+1 \\
 &   & w_k \mathbf{1} & u_{k+1} M & n-2^k-1 \\[-1em]  \\ %tiny extra row to avoid overlap
w_k \mathbf{1}^T  & w_k \mathbf{1}^T  & u_{k+1} & \mathbf{0} &   1 \\
\mathbf{0} & u_{k+1} M &  \mathbf{0}  & \mathbf{0} &  n-2^k-1 
\CodeAfter
\SubMatrix[{1-1}{5-4}]
\end{NiceArray}.\]
By expanding the determinant of this matrix with respect to the last column of blocks and then the last row of blocks, we get
$$ \mathcal{H}_{\mathbf{w}}(n) = (-1)^{n-2^k-1} u_{k+1}^{2n-2^{k+1}-2}\det \begin{bmatrix}
     H_{\mathbf{w}}(2^{k+1}-n+1) & w_k \mathbf{1} \\
     w_k \mathbf{1}^T & u_{k+1} 
\end{bmatrix}.  $$
Now, for each $i=2^{k+1}-n-1, 2^{k+1}-n-2, \ldots, 1,0$, in this order, we subtract the $i$th column from the $(i+1)$st one. We get a matrix of the form
$$ \begin{bmatrix}
     V & W & w_k \mathbf{1} \\
     w_k & \mathbf{0} & u_{k+1} 
\end{bmatrix}, $$
where $V$ is a column vector of length $2^{k+1}-n+1$ and $W$ is a  matrix of size $(2^{k+1}-n+1)\times(2^{k+1}-n)$. Observe that $\begin{bmatrix} V & W
\end{bmatrix}$ is column-equivalent to $H_{\mathbf{w}}(2^{k+1}-n+1)$, while $\begin{bmatrix}
    W & \mathbf{1}
\end{bmatrix}=\widetilde{G}_{\mathbf{w}}(2^{k+1}-n+1)$. 
After expanding with respect to the last row, the remaining determinant becomes
$$ u_{k+1} \mathcal{H}_{\mathbf{w}}(2^{k+1}-n+1) +(-1)^n w_k^2 \widetilde{\mathcal{G}}_{\mathbf{w}}(2^{k+1}-n+1).
$$
The formula in (a) follows shortly.

We move on to case (b), where the considered $n \times n$ submatrix takes the form
\[\begin{NiceArray}{ccccccl}[right-margin=.4em]
\Block[borders={bottom,right,tikz=dashed}]{3-3}<\Large>{H_{\mathbf{w}}(2^k)} &  &  & w_k \mathbf{1} & \mathbf{0} & \mathbf{0} & 2^{k+1}-n+1 \\
 &  &  & w_k \mathbf{1} & \mathbf{0} & u_{k+1}M & 2n -3 \cdot 2^k -1 \\
 &  &  & w_k \mathbf{1} & u_{k+1} M & u_{k+1} J&  2^{k+1}-n \\[-1em]  \\ %tiny extra row to avoid overlap
w_k \mathbf{1}^T & w_k \mathbf{1}^T & w_k \mathbf{1}^T & u_{k+1} & \mathbf{0} & \mathbf{0} & 1 \\
\mathbf{0} & \mathbf{0} &  u_{k+1} M & \mathbf{0} & \mathbf{0} & \mathbf{0} & 2^{k+1}-n \\
\mathbf{0} & u_{k+1}M & u_{k+1}J &  \mathbf{0} & \mathbf{0} & -2u_{k+1} M & 2n - 3 \cdot 2^k-1
\CodeAfter
\SubMatrix[{1-1}{7-6}]
\end{NiceArray}.\]

By expanding with respect to the penultimate row and column of blocks, we get
\begin{align*}
\mathcal{H}_{\mathbf{w}}(n) = &(-1)^{2^{k+1}-n} u_{k+1}^{2^{k+2}-2n} \\
&\det \begin{NiceArray}{ccccl}[left-margin=.4em,right-margin=.4em]
\Block[borders={bottom,right,tikz=dashed}]{2-2}<\Large>{H_{\mathbf{w}}(n-2^k)} &   & w_k \mathbf{1} &  \mathbf{0} & 2^{k+1}-n+1 \\
 &  &  w_k \mathbf{1} &  u_{k+1}M & 2n -3 \cdot 2^k -1\\[-1em]  \\ %tiny extra row to avoid overlap
w_k \mathbf{1}^T &  w_k \mathbf{1}^T & u_{k+1} & \mathbf{0} & 1 \\
\mathbf{0} & u_{k+1}M &   \mathbf{0} & -2u_{k+1} M & 2n - 3 \cdot 2^k-1
\CodeAfter
\SubMatrix[{1-1}{5-4}]
\end{NiceArray}
\end{align*}.
We now add the last row of blocks multiplied by $1/2$ to the second row. Then $u_{k+1}/2 = w_{k+1}/2 - w_k$ is added at the positions in $H_{\mathbf{w}}(n-2^k)$ where $w_k$ appears, i.e., $w_k$ is replaced with $w_k' = w_{k+1}/2$. After this, we get
\[ \mathcal{H}_{\mathbf{w}}(n) = (-1)^n  u_{k+1}^{2^{k+2}-2n} \cdot (-2u_{k+1})^{2n-3 \cdot 2^k-1} \det \begin{NiceArray}{ccl}[left-margin=.4em,right-margin=.4em]
H_{\mathbf{w}'}(n-2^k) &  w_k \mathbf{1} &  n-2^k \\
   w_k \mathbf{1}^T & u_{k+1}   & 1 
   \CodeAfter
\SubMatrix[{1-1}{2-2}]
\end{NiceArray}, \]
and the remaining computations are similar as in part (a).
\end{proof}

We now state a polynomial specialization of Theorem \ref{thm:det_H_recurrence}, which will be  useful when computing and studying the properties of Hankel determinants $\mathcal{H}(n,t)$.  We introduce auxiliary polynomials in variables $t,x$, defined for $2^k < n \leq 2^{k+1}$ by
\begin{align*}
    h(n,t,x) &:= \mathcal{H}_{\mathbf{w}}(n),  \\ \widetilde{g}(n,t,x) &:= \widetilde{\mathcal{G}}_{\mathbf{w}}(n), \\
    \text{where } \mathbf{w} &= (1,t,\ldots,t^k,x).
\end{align*}
In particular, we get 
$$
   \mathcal{H}(n,t)= h(n,t,t^{k+1}), \qquad \mathcal{H}(n,1)=h(n,1,1).
$$
Theorem \ref{thm:det_H_recurrence} then takes the following form.

\begin{prop} \label{prop:h_g_recurrence}
    Let  $n \in \N_{\geq 5}$ and $k \in \N$ be such that $2^k < n \leq 2^{k+1}$. 
The following recursive formulas hold.
\begin{enumerate}[label=(\alph*)]
    \item If $2^k < n \leq 3 \cdot 2^{k-1}$, then
$$\begin{bmatrix}
    h(n,t,x) \\
    \widetilde{g}(n,t,x)
\end{bmatrix} = 
(-1)^n (x-2t^k)^{2n-2^{k+1}-1} \begin{bmatrix}
    -1 & \frac{(-1)^n t^{2k}}{x-2t^k} \\ 0 & 1
\end{bmatrix}
\begin{bmatrix}
   h(2^{k+1}-n+1,t,t^k) \\
\widetilde{g}(2^{k+1}-n+1,t,t^k)
\end{bmatrix}.
$$
     \item If $3 \cdot 2^{k-1} < n \leq 2^{k+1}$, then
    $$\begin{bmatrix}
    h(n,t,x) \\
    \widetilde{g}(n,t,x)
\end{bmatrix}  = 
2^{2n-3\cdot 2^k-1} (x-2t^k)^{2^k} \begin{bmatrix}
    1 & \frac{(-1)^n t^{2k}}{x-2t^k} \\ 0 & 1
\end{bmatrix}
\begin{bmatrix}
    h(n-2^k,t,\frac{x}{2}) \\
    \widetilde{g}(n-2^k,t,\frac{x}{2})
\end{bmatrix}.
$$
  \end{enumerate}
\end{prop}

In the remainder of this section, we explore various implications of Theorem \ref{thm:det_H_recurrence} and Proposition \ref{prop:h_g_recurrence}.

\subsection{Special cases}
We start with the aforementioned polynomial generalization of \eqref{eq:simple_expression}. As before, we consider the sequence of indices $n_k = \lceil 2^{k+2}/3 \rceil$, which can be equivalently defined by $n_0 = 2$ and
\begin{equation} \label{eq:n_k}
    n_k = 2^{k+1}-n_{k-1}+1, \quad k \geq 1,
\end{equation} 
or $n_k = 2n_{k-1} - (k \bmod{2})$.

\begin{thm} \label{thm:nice_formula}
    For all $k \in \N$ we have
    $$ \mathcal{H}(n_k,t) = (-1)^{\frac{(k+1)(k+2)}{2}} t^{\alpha_k} (t-2)^{n_k-2} (1+t+ \cdots + t^k),$$
    where
  
    $$
    \alpha_k = (k-1)n_k+2 - \left\lfloor \frac{k+1}{2} \right\rfloor.
    $$
    In particular, Hankel determinants for the usual binary sum of digits satisfy
    $$ \mathcal{H}(n_k,1) = (-1)^{\frac{(k+2)(k+3)}{2}} (k+1).$$
\end{thm}
\begin{proof}
We use induction on $k$ to simultaneously prove the expression for $\mathcal{H}(n_k,t)$ and
$$\widetilde{\mathcal{G}}(n_k,t) = (-1)^{\frac{(k+2)(k+3)}{2}} \: t^{\alpha_k} (t-2)^{n_k-1}. $$
Our claim holds for $k=0,1$ by direct computation. 

Now, let $k \geq 2$ and assume that both formulas are true for $k-1$. Since $ 2^k < n_k \leq 3 \cdot 2^{k-1}$, by Theorem \ref{thm:det_H_recurrence}(a) we get
$$
\begin{bmatrix}
    \mathcal{H}(n_k,t) \\
    \widetilde{\mathcal{G}}(n_k,t)
\end{bmatrix} =
(-1)^{n_k} \left(t^k(t-2)\right)^{n_k-n_{k-1}}\begin{bmatrix}
    -1 & (-1)^{n_k} \frac{t^k}{t-2}\\
    0 & 1
\end{bmatrix}
\begin{bmatrix}
    \mathcal{H}(n_{k-1},t) \\
    \widetilde{\mathcal{G}}(n_{k-1},t)
\end{bmatrix}.
$$
Using the inductive assumption together with $n_k \equiv k \pmod{2}$ and  $\alpha_k =  \alpha_{k-1} + k(n_k - n_{k-1})$, we get the formulas for $k$.
\end{proof}

The main reason behind the existence of such a nice formula seems to be the fact that at each step of the induction, we used part (a) of Theorem \ref{thm:det_H_recurrence} so the weights $w_j=t^j$ did not change at any point.

In the next result, we give more complicated, although still compact formulas for $\mathcal{H}(2^m+\varepsilon,1)$ for $\varepsilon\in\{-1,0,1\}$.

\begin{prop} \label{prop:special_cases}
    For any $m \in \N_{\geq 2}$ we have
    $$\mathcal{H}(2^m,1) =  \sum_{j=1}^m \frac{2^{j-1}}{2^j-1} \prod_{j=1}^m (2^j-1)^{2^{m-j}},$$
    $$
    \mathcal{H}(2^m+1,1) =  -\left(1+\sum_{j=1}^m \frac{2^{j-1}}{2^j-1} \right)\prod_{j=1}^m (2^j-1)^{2^{m-j}},
    $$
    and
    $$ \mathcal{H}(2^m-1,1) = (2^{m-1}-1) \left(1 + \sum_{j=1}^{m-1} \frac{2^{j-1}}{2^{j}-1}\right) \prod_{j=1}^{m-2} (2^j-1)^{2^{m-j}} . $$
\end{prop}
\begin{proof}
All three formulas are valid for $m=2$, so assume $m \geq 3$.
By repeated application of Proposition \ref{prop:h_g_recurrence}(b), for $k \geq 2$ and $l \geq 1$ we get
\begin{equation} \label{eq:h_2kl}
\begin{bmatrix}
    h(2^{k+l},1,x) \\
    \widetilde{g}(2^{k+l},1,x) 
\end{bmatrix} = 
\frac{1}{2^l} \prod_{j=1}^l (2^{2-j}x-4)^{2^{k+l-j}} \begin{bmatrix}
   1  & \sum\limits_{j=0}^{l-1} \frac{1}{x/2^j-2} \\ 0 & 1
\end{bmatrix}
\begin{bmatrix}
    h(2^k,1,\frac{x}{2^l}) \\
    \widetilde{g}(2^k,1,\frac{x}{2^l}) 
\end{bmatrix}.
\end{equation}
Letting $k=2$, $l=m-2$, and $x=1$, we thus obtain
\begin{align*}
    h(2^m,1,1) = &\frac{1}{2^{m-2}}  \prod_{j=1}^{m-2} (2^{2-j}-4)^{2^{m-j}} \\ 
             &\left[h\left(4,1,\frac{1}{2^{m-2}}\right) + \widetilde{g}\left(4,1,\frac{1}{2^{m-2}}\right)\sum\limits_{j=0}^{m-3} \frac{2^{j-1}}{2^j-1} \right].
\end{align*} 
Taking into account that $h(4,1,y) = (y-2)(3y-8)$ and $\widetilde{g}(4,1,y) = (y-2)^2(y-4),$
after some algebraic manipulation we get the desired formula
\begin{equation} \label{eq:h_2m}
    \mathcal{H}(2^m,1) = h(2^m,1,1) =\sum_{j=1}^m \frac{2^{j-1}}{2^j-1} \prod_{j=1}^m (2^j-1)^{2^{m-j}}.
\end{equation}

To obtain the formula for $\mathcal{H}(2^m+1,1)$, we apply Proposition \ref{prop:h_g_recurrence}(a) to $k=m,  n=2^m+1$, and $t=x=1$, obtaining
\begin{equation} \label{eq:H_2m_plus_1_reduction}
h(2^m+1,1,1)= -h(2^m,1,1)+g(2^m,1,1).
\end{equation}
From \eqref{eq:h_2kl} we get
$$
g(2^m,1,1)= \frac{1}{2^{m-2}}
\prod_{j=1}^{m-2}(2^{2-j}-4)^{2^{m-j}}
\, g\!\left(4,1,\frac{1}{2^{m-2}}\right) = -\prod_{j=1}^m (2^j-1)^{2^{m-j}}.
$$
Substituting this and \eqref{eq:h_2m} into
\eqref{eq:H_2m_plus_1_reduction}, we get
$$
h(2^m+1,1,1)
=
-\left(1+\sum_{j=1}^m \frac{2^{j-1}}{2^j-1}\right)
\prod_{j=1}^m (2^j-1)^{2^{m-j}}.
$$
The proof of the last identity relies on the formula 
$$\begin{bmatrix}
    h(2^{k+l}-1,1,x) \\
    \widetilde{g}(2^{k+l}-1,1,x) 
\end{bmatrix} = 
\frac{1}{2^l} \prod_{j=1}^l (2^{2-j}x-4)^{2^{k+l-j}} \begin{bmatrix}
   1  & \sum\limits_{j=0}^{l-1} \frac{1}{x/2^j-2} \\ 0 & 1
\end{bmatrix}
\begin{bmatrix}
    h(2^k-1,1,\frac{x}{2^l}) \\
    \widetilde{g}(2^k-1,1,\frac{x}{2^l}) 
\end{bmatrix},$$
which again follows from Proposition \ref{prop:h_g_recurrence}(b). The details are left for the reader to verify.
\end{proof}

Based on the form of the values of $\mathcal{H}(2^m+k,1)$ for $k=-1, 0, 1$ given in Proposition \ref{prop:special_cases}, we formulate the following conjecture.

\begin{conj}
For every fixed integer $k\ge 0$ there exist integers $m_0=m_0(k)\ge 0$ and
$r=r(k)\ge 0$, a sign $\varepsilon_k\in\{\pm 1\}$, a polynomial
$P_k(X)\in \mathbb{Q}[X]$, and a constant $c_k\in \mathbb{Q}$ such that for all
$m\ge m_0(k)$ we have
\[
\mathcal H(2^m+k,1)
=
\varepsilon_k\, P_k(2^m)
\left(
c_k+\sum_{j=1}^{m-r}\frac{2^{j-1}}{2^j-1}
\right)
\prod_{j=1}^{m-r}(2^j-1)^{2^{m-r-j}}.
\]
Moreover, the polynomial $P_k(2^m)$ and the shift $r(k)$ are determined by the
binary expansion of $k$ and the sequence of ``branches'' encountered when repeatedly applying Theorem \ref{thm:det_H_recurrence}.
\end{conj}

Another interesting feature of the sequence $(\mathcal{H}(n,1))_{n \geq 1}$, suggested by experimental computations, is that the sequence of absolute values $(|\mathcal{H}(n,1)|)_{n \geq 2}$ contains long monotone increasing and decreasing subsequences. Moreover, local maxima and minima seem to occur at the indices $2^k+1$ and $n_k$ (defined by $\eqref{eq:n_k}$), respectively. Motivated by this observation, we formulate the following conjecture.

\begin{conj}
    The sequence of absolute values $(|\mathcal{H}(n,1)|)_{n \geq 2}$ is:
    \begin{enumerate}
        \item strictly increasing in each interval $[n_k,2^{k+1}+1]$;
        \item strictly decreasing in each interval $[2^{k+1}+1,n_{k+1}]$.
    \end{enumerate}
    In particular, for each $k \in \N$ we have
    $$ |\mathcal{H}(2^k+1,1)| = \max_{n \leq 2^{k}+1} |\mathcal{H}(n,1)|. $$
\end{conj}

\subsection{Degree and leading coefficient}

We now study the polynomials $\mathcal{H}(n,t)$ in terms of their degree and leading coefficient. In fact, we give a more general result concerning the weighted degree of $\mathcal{H}_\mathbf{w}(n)$.

Let $P$ be a polynomial in variables $w_0, w_1, \ldots, w_l$:
    $$P(w_0,w_1,\ldots,w_l) = \sum_{a=(a_0,a_1,\ldots,a_l)\in \mathcal{A}} c_a w_0^{a_0} w_1^{a_1} \cdots w_l^{a_l}, $$
    where $\mathcal{A} \subset \N^{l+1}$ is finite and  $c_a \in \C \setminus \{0\}$. 
   Its weighted degree $D(P)$, where $w_j$ is assigned weight $j$, is defined by
   $$D(P) :=  \max\left\{ \sum_{j=0}^l j a_j: \:  a \in \mathcal{A} \right\}.$$
   If there is precisely one monomial satisfying $D(c_a w_0^{a_0} w_1^{a_1} \cdots w_l^{a_l}) = D(P)$ we will say that it is the leading monomial of $P$ (this is not standard terminology).

Now, define
$$  D_n := D(\mathcal{H}_\mathbf{w}(n)). $$
   The weights are chosen in such a way that  $\deg_t \mathcal{H}(n,t)$ is expected to be equal to $D_n$.   
   The equality is not obvious due to potential cancellation in $\mathcal{H}(n,t)$, however, the following proposition shows that it indeed holds. 

\begin{prop} \label{prop:degree}
For $n \in \N_{\geq 2}$ we have the following:
\begin{enumerate}[label=(\alph*)]
\item for $k \in \N$ such that $2^k < n \leq 2^{k+1}$ we have
$$ D_n =\begin{cases}
   D_{2^{k+1}-n+1} +  (2n-2^{k+1}-1)(k+1) +1 &\text{if } n \leq 3 \cdot 2^{k-1}, \\
    D_{n-2^k} +2n+2^k(k-2) &\text{if } n > 3 \cdot 2^{k-1};
\end{cases}   $$
\item the leading coefficient of $\mathcal{H}(n,t)$ belongs to $\{1,-1\}$;
\item $\deg_t \mathcal{H}(n,t) = D_n$.
\end{enumerate}
\end{prop}
\begin{proof}
Let $E_n := D(\widetilde{\mathcal{G}}_\mathbf{w}(n))$.
   Let $n \geq 2$ and $k$ be such that $2^k < n \leq 2^{k+1}$. We simultaneously prove the following statements by induction:
   \begin{enumerate}[label=(\roman*)]
       \item $D_n = E_n +k-1$;
       \item $\mathcal{H}_\mathbf{w}(n)$ has a leading monomial and the coefficient of this monomial is $\pm1$.
   \end{enumerate}
  Both statements hold for $n < 5$, and thus let $n \geq 5$. We will use the notation $A_{\mathbf{w}}(n), B_{\mathbf{w}}(n), C_{\mathbf{w}}(n)$ from Theorem \ref{thm:det_H_recurrence}. In particular, we have
  \begin{align*}
      D(A_{\mathbf{w}}(n)) &= (2n-2^{k+1}-1)(k+1) \\
      D(A_{\mathbf{w}}(n) C_{\mathbf{w}}(n)) &= (2n-2^{k+1})(k+1)-2, \\
      D(B_{\mathbf{w}}(n)) &= 2^k(k+1), \\
      D(B_{\mathbf{w}}(n) C_{\mathbf{w}}(n)) &= (2^k+1)(k+1)-2.
  \end{align*}  
  If $2^k < n \leq 3 \cdot 2^{k-1}$, then $2^{k-1} < 2^{k+1}-n+1 \leq 2^k$ and the inductive assumption implies
  $$  D(A_{\mathbf{w}}(n) \mathcal{H}_\mathbf{w}(2^{k+1}-n+1)) = D(A_{\mathbf{w}}(n) C_{\mathbf{w}}(n) \widetilde{\mathcal{G}}_\mathbf{w}(2^{k+1}-n+1))-1. $$
  Hence, by Theorem \ref{thm:det_H_recurrence}(a) we get
  $$    D_n = D(A_{\mathbf{w}}(n) C_{\mathbf{w}}(n) \widetilde{\mathcal{G}}_\mathbf{w}(2^{k+1}-n+1)) = E_n + k-1. $$
  Combining these two equalities, we also obtain the first case in (a).

  To see that (ii) holds, note that by Proposition \ref{prop:det_Gt}, the polynomial $\widetilde{\mathcal{G}}_\mathbf{w}(2^{k+1}-n+1)$ has a leading monomial with coefficient $\pm 1$. This remains true for the polynomial $A_{\mathbf{w}}(n) C_{\mathbf{w}}(n) \widetilde{\mathcal{G}}_\mathbf{w}(2^{k+1}-n+1)$, and its leading monomial becomes the leading monomial in $\mathcal{H}_\mathbf{w}(n)$ (up to sign).
 
  In the case $3 \cdot 2^{k-1} < n \leq 2^{k+1}$ we have $2^{k-1} < n-2^k \leq 2^k$. We need a bit more caution when comparing weighted degrees of the polynomials $B_{\mathbf{w}}(n) \mathcal{H}_\mathbf{w'}(n-2^k)$ and $B_{\mathbf{w}}(n) C_{\mathbf{w}}(n) \widetilde{\mathcal{G}}_\mathbf{w'}(n-2^k)$ appearing in Theorem \ref{thm:det_H_recurrence}(b). By Proposition \ref{prop:det_Gt}, the leading monomial in $\widetilde{\mathcal{G}}_\mathbf{w}(n-2^k)$ contains $w_k$ with multiplicity $2n-3 \cdot 2^k-1$. Substituting $w_k \mapsto \frac{1}{2} w_{k+1}$ (which corresponds to $\mathbf w \mapsto \mathbf{w}'$) maps the leading monomial in $\widetilde{\mathcal{G}}_\mathbf{w}(n-2^k)$ to the leading monomial in $\widetilde{\mathcal{G}}_\mathbf{w'}(n-2^k)$. It follows that
  $$ D(\widetilde{\mathcal{G}}_\mathbf{w'}(n-2^k)) = E_{n-2^k} + 2n-3 \cdot 2^k-1. $$
  At the same time, $w_k$ occurs in $\mathcal{H}_\mathbf{w}(n-2^k)$ precisely at positions $(i,j)$ such that $i+j \geq 2^k$, which are contained in a square submatrix of size $2n-3 \cdot 2^k-1$. Hence, after the substitution $w_k \mapsto \frac{1}{2} w_{k+1}$
  we get
  $$ D(\mathcal{H}_\mathbf{w'}(n-2^k)) \leq D_{n-2^k} + 2n-3 \cdot 2^k-1. $$
  Together with the inductive assumption, this leads to
  $$ D(B_{\mathbf{w}}(n) \mathcal{H}_\mathbf{w'}(n-2^k)) < D(B_{\mathbf{w}}(n) C_{\mathbf{w}}(n) \widetilde{\mathcal{G}}_\mathbf{w'}(n-2^k)).  $$
 The remainder of the argument is similar to that before.

 Parts (b) and (c) follow directly from the statement (ii), which ensures that there is no cancellation of terms of degree $D_n$ after substituting $w_j \mapsto t^j$.
\end{proof}

As it turns out, there is a curious connection between the sequence of degrees $D_n$ and sums of sums of digits.  
For $n \in \N$ we let 
$$ S_n = \sum_{\ell=0}^n s(\ell), $$
where $s$ denotes the sum of binary digits of $n$.

\begin{cor}
Let $n \in \N_{\geq 2}$ and $k \in \N$ be such that $2^k < n \leq 2^{k+1}$. Then 
    $$  D_n = S_{n-1}+S_{n-2}+k-1. $$
\end{cor}
\begin{proof}
Put $D_n' = S_{n-1}+S_{n-2} + k-1$ for  $2^k < n  \leq 2^{k+1}$. We have $D_2 = D_2' = 0$ so it suffices to show that $D_n'$ satisfies the same recurrence relations as in Proposition \ref{prop:degree}(a).

We recall two identities satisfied by $S_n$. Using the symmetry
$$ s(\ell)+s(2^{k+1}-1-\ell)=k+1, \qquad 0\leq \ell<2^{k+1},$$
for $m < 2^{k+1}-1$, one obtains
\begin{equation}\label{eq:S_identity1}
    S_m = S_{2^{k+1}-m-2} + (k+1)(m-2^k+1).
\end{equation}
Moreover, for $m < 2^k$ we have
\begin{equation}\label{eq:S_identity2}
  S_{2^k+m} = S_{2^k-1} + \sum_{\ell=0}^m s(2^k+\ell)  =  k 2^{k-1} + (m+1) + S_m.
\end{equation}

Now, suppose that  $2^k < n \leq 3 \cdot 2^{k-1}$. Applying \eqref{eq:S_identity1} to $m=n-1, n-2$, we get
\begin{align*}
    D_n' &= S_{2^{k+1}-n-1}+S_{2^{k+1}-n}+ (k+1)(2n-2^{k+1}-1)+k-1 \\
    &= D_{2^{k+1}-n+1}' + (k+1)(2n-2^{k+1}-1) + 1.
\end{align*} 
If $3 \cdot 2^{k-1} < n \leq 2^{k+1}$, then \eqref{eq:S_identity2} applied to $m=n-2^k-1, n-2^k-2$, gives
$$ D_{n}' = k 2^k + 2n - 2^{k+1}-1 + S_{n-2^k-1} + S_{n-2^k-2} + k-1 = 2n + 2^k(k-2) +D_{n-2^k}',  $$
which ends the proof.
\end{proof}

We do not know whether this identity is purely coincidental or has a deeper reason.

\subsection{Factorization of $\mathcal{H}(n,t)$}
In the next theorem, we describe the general structure of the polynomials $\mathcal{H}(n,t)$. It turns out that after dividing $\mathcal{H}(n,t)$ by the factors $t$ and $t-2$ (with multiplicity), all coefficients of the remaining factor have a common sign.

\begin{thm} \label{thm:factorization}
    For $n \in \N_{\geq 2}$ we have
    $$ \mathcal{H}(n,t) = \sigma(n) t^{\beta(n)} (2-t)^{n-2}\hat{\mathcal{H}}(n,t), $$
    where $\hat{\mathcal{H}}(n,t) \in \Z[t]$ is a monic polynomial with all coefficients positive, and $\sigma(n) \in \{1,-1\}$, $\beta(n) \in \N$ satisfy 
    \begin{align*}
        \sigma(2) &=-1, \\
        \sigma(n) &= \begin{cases}
    (-1)^n \sigma(2^{k+1}-n+1) &\text{if } 2^k < n \leq 3 \cdot 2^{k-1}, \\
   \sigma(n-2^k) &\text{if } 3 \cdot 2^{k-1} < n \leq 2^{k+1},
\end{cases}  \\
\beta(2) &= 0, \\
\beta(n) &= \begin{cases}
    \beta(2^{k+1}-n+1)+k(2n-2^{k+1}-1) &\text{if } 2^k < n \leq 3 \cdot 2^{k-1}, \\
   \beta(n-2^k) + k 2^k &\text{if } 3 \cdot 2^{k-1} < n \leq 2^{k+1},
\end{cases}
    \end{align*}
    
\end{thm}
\begin{proof}
We prove by induction that for $2^k < n \leq 2^{k+1}$ the following two statements hold:
\begin{enumerate}[label=(\roman*)]
    \item for all $l \geq 1$ there is a factorization
    \begin{align*}
h \left(n,t,\frac{t^{k+l}}{2^{l-1}}\right) &= \sigma(n) t^{\beta_n} (2-t)^{n-2} \hat{h}_l(n,t), \\
\widetilde{g} \left(n,t,\frac{t^{k+l}}{2^{l-1}}\right) &= (-1)^{n+1} \sigma(n) t^{\beta_n} (2-t)^{n-1} \hat{g}_l(n,t),
\end{align*}
where $\hat{h}_l(n,t),\hat{g}_l(n,t)  \in \Q[t]$ have all coefficients non-negative;
\item $\deg \hat{h}_l(n,t) \geq k$.
\end{enumerate}
Along the way, we will also obtain the relations for $\sigma(n)$ and $\beta(n)$.

We are going to use the notation
$$  r_l(t) = t^{l-1} + 2t^{l-2} + \cdots + 2^{l-2}t + 2^{l-1} = \frac{t^l-2^l}{t-2}. $$
For $n=2,3,4$ our claim holds because
\begin{align*}
    h \left(2,t,\frac{t^{l}}{2^{l-1}}\right) &=  -1,\\
    h\left(3,t,\frac{t^{1+l}}{2^{l-1}}\right) &=  2 t + 2 t^2 - t^3 - \frac{t^{1+l}}{2^{l-1}} =  t(2-t)\left(t+ \frac{1}{2^{l-1}}r_l(t) \right),\\
    h\left(4,t,\frac{t^{1+l}}{2^{l-1}}\right) &=  -\left(2 t-\frac{t^{1+l}}{2^{l-1}}\right) \left(\frac{t^{3+l}}{2^{l-1}}+ \frac{t^{1+l}}{2^{l-2}}-4 t^2-4 t\right) \\
    &= \frac{1}{2^{2l-3}} t^2 (2-t)^2  r_l(t) \Bigl(2r_l(t) +t r_{l+1}(t) \Bigr),
\end{align*}
and
\begin{align*}
    \widetilde{g} \left(2,t,\frac{t^{l}}{2^{l-1}}\right) &= 2- \frac{t^{l}}{2^{l-1}} = \frac{1}{2^{l-1}} (2-t) r_l(t),\\
    \widetilde{g}\left(3,t,\frac{t^{1+l}}{2^{l-1}}\right) &=  (2-t)\left(2t- \frac{t^{1+l}}{2^{l-1}} \right) = \frac{1}{2^{l-1}} t(2-t)^2 r_l(t),\\
    \widetilde{g}\left(4,t,\frac{t^{1+l}}{2^{l-1}}\right) &=  \left(2t-\frac{t^{1+l}}{2^{l-1}}\right)^2 \left(\frac{t^{1+l}}{2^{l-1}}-4 \right) \\
    &= -\frac{1}{2^{2l-2}} t^2 (2-t)^3 r_{l}(t) r_{l+1}(t).
\end{align*}

Hence, assume that $n \geq 5$. 
 In the case $2^k < n \leq 3 \cdot 2^{k-1}$ put $n'= 2^{k+1}-n+1 \in (2^{k-1},2^k]$. Then Proposition \ref{prop:h_g_recurrence}(a) together with the inductive assumption gives
 \begin{equation} \label{eq:h_nli}
\begin{aligned} 
     h \left(n,t,\frac{t^{k+l}}{2^{l-1}}\right) &=
     (-1)^n  \left(\frac{t^{k+l}}{2^{l-1}} - 2t^k \right)^{2n-2^{k+1}-2} \times\\
     &\phantom{=}\left[-\left(\frac{t^{k+l}}{2^{l-1}} - 2t^k\right)h(n',t,t^k) + (-1)^n t^{2k} \widetilde{g}(n',t,t^k)\right] \\
     &= (-1)^n  \sigma(n') t^{k(2n-2^{k+1}-1)+\beta(n')} (2-t)^{2n-2^{k+1}+n'-3}  \times\\ &\phantom{=}\left(\frac{r_l(t)}{2^{l-1}}  \right)^{2n-2^{k+1}-2}  \left[\frac{r_l(t)}{2^{l-1}}   \hat{h}_1(n',t) + (-1)^{n+n'+1}t^k \hat{g}_1(n',t) \right].
 \end{aligned}  
 \end{equation}
 Since $n+n'+1$ is even, we get the polynomial
 $$  \hat{h}_l(n,t) = \left(\frac{r_l(t)}{2^{l-1}}   \right)^{2n-2^{k+1}-2} \left[\frac{r_l(t)}{2^{l-1}}   \hat{h}_1(n',t) + t^k \hat{g}_1(n',t) \right], $$
which has degree $\geq k$ and all coefficients positive. From \eqref{eq:h_nli} we also get the desired relations for $\sigma(n)$, $\beta(n)$ and note that the exponent of $(2-t)$ is $n-2$. Statement (i) for $\widetilde{g}$ is proved in a similar way.

 In the case $3 \cdot 2^{k-1} < n \leq 2^{k+1}$ put $n'= n-2^k$ so that by Proposition \ref{prop:h_g_recurrence}(b) we get
 \begin{align*}
     h \left(n,t,\frac{t^{k+l}}{2^{l-1}}\right) &=
    2^{2n-3 \cdot 2^k-1} \left(\frac{t^{k+l}}{2^{l-1}} - 2t^k \right)^{2^k-1} \times\\
     &\phantom{=}\left[\left(\frac{t^{k+l}}{2^{l-1}} - 2t^k\right)h\left(n',t,\frac{t^{k+l}}{2^{l}}\right) + (-1)^n t^{2k} \widetilde{g}\left(n',t,\frac{t^{k+l}}{2^{l}}\right)\right] \\
     &=  \sigma(n') t^{k 2^k + \beta(n')} (2-t)^{2^k+n'-2}  \times\\
     &\phantom{=} 2^{2n-3 \cdot 2^k-1} \left(\frac{r_l(t)}{2^{l-1}}  \right)^{2^k-1}
      \left[ \frac{r_l(t)}{2^{l-1}}   \hat{h}_{l+1}(n',t) + (-1)^{n+n'}t^k \hat{g}_{l+1}(n',t) \right].
  \end{align*}  
 The polynomial in the last line  is $\hat{h}_l(n,t)$, which again has the desired properties. As before, the computation for $\widetilde{g}$ is analogous, and the induction is finished.

 Now, in the main statement of the proposition we take
 $$ \hat{\mathcal{H}}(n,t) = \hat{h}_1(n,t), $$
 which is a polynomial with all coefficients positive, and belongs to $\Z[t]$ because $\mathcal{H}(n,t) \in \Z[t]$. Moreover, it is monic due to Proposition \ref{prop:degree}, which ends the proof.
 \end{proof}

We note that the factor $\hat{\mathcal{H}}(n,t)$ in the proposition is usually reducible. One can check that for given $l \geq 2$ we often have the divisibility $\frac{t^l-2^l}{t-2} \mid \hat{\mathcal{H}}(n,t)$ (with some multiplicity).  Also, Theorem \ref{thm:nice_formula} implies that $\frac{t^{k+1}-1}{t-1} \mid \hat{\mathcal{H}}(n_k,t)$. In Section \ref{sec:vanishing}, we study this in more detail.

\subsection{Sign behavior}

We now take a closer look at the sequence of signs $\sigma(n)$ introduced in Theorem \ref{thm:factorization}. In particular, the proposition implies that Hankel determinants for the usual binary sum of digits are nonzero.

\begin{cor}
    For all $n \in \N_{\geq 2}$ the Hankel determinant $\mathcal{H}(n,1)$ is nonzero and we have
    $$  \sgn \mathcal{H}(n,1) = \sigma(n). $$
\end{cor}

The signs satisfy another, arguably simpler set of recurrence relations, where for convenience we define $\sigma(1)=1$.
\begin{prop}
For all $n \in \N_+$ we have
\begin{equation} \label{eq:sigma}
    \begin{aligned}
    \sigma(4n) &= (-1)^{n} \sigma(2n), \\
    \sigma(4n+1) &= (-1)^n \sigma(2n+1),\\
    \sigma(4n+2) &= (-1)^{n+1} \sigma(2n+1), \\    
    \sigma(4n+3) &= -\sigma(2n+1). 
\end{aligned}
\end{equation}
\end{prop}
\begin{proof}
We use induction on $n$. For $n=1,2$ our claim holds, and thus let $n \geq 3$. The argument is similar for each relation, so we only prove the second one.
If $2^k < 4n+1 \leq 3 \cdot 2^{k-1}$, then by Theorem \ref{thm:factorization}  and the inductive assumption, we get
$$
    \sigma(4n+1) = -\sigma(2^{k+1}-4n) = (-1)^{n+1} \sigma(2^{k}-2n) = (-1)^n \sigma(2n+1)
$$
If $3 \cdot 2^{k-1} < 4n +1 \leq 2^{k+1}$, then again
$$ \sigma(4n+1) = \sigma(4n- 2^{k}+1) = (-1)^n \sigma(2n-2^{k-1}+1) =  (-1)^n \sigma(2n+1),$$
as desired.
\end{proof}

We now show that the ordinary generating function of the sequence $(\sigma(n))_{n \geq 1}$ of signs of $\mathcal{H}(n,1)$ is transcendental. In particular, this implies that this sequence is not periodic. 

\begin{thm}
Let $F(x)=\sum_{n=1}^{\infty}\sigma(n)x^{n}$. The function $F(x)$ is transcendental over $\Q(x)$. In particular, the sequence $(\sigma(n))_{n \geq 1}$ is not periodic.
\end{thm}
\begin{proof}
It is clear that to show that the function $F(x)$ is transcendental it is 
enough to show that its odd part 
$$
C(x)=\sum_{n=0}^{\infty}\sigma(2n+1)x^{n}
$$
is transcendental. We split the series according to the parity of $n$:
$$
C(x)=\sum_{n=0}^{\infty} \sigma(4n+1)\, x^{2n}+\sum_{n=0}^{\infty} \sigma(4n+3)\, x^{2n+1},
$$
and using the relations $\sigma(4n+1)=(-1)^n \sigma(2n+1)$ and $\sigma(4n+3)=-\sigma(2n+1)$, we obtain the functional relation
\begin{equation}\label{functrel1}
C(x)=\sum_{n\geq 0} (-1)^n \sigma(2n+1) x^{2n}- x\sum_{n\geq 0} \sigma(2n+1) x^{2n}= C(-x^2)-xC(x^2).
\end{equation}
Next, in order to eliminate the part with $C(-x^2)$ from \eqref{functrel1}, we replace $x$ by $-x$ and get the relation 
$C(-x)=C(-x^2)+xC(x^2)$, or equivalently $C(-x^2)=C(-x)-xC(x^2)$. Putting this expression into (\ref{functrel1}) leads us to the identity
$$
C(-x)=2xC(x^2)+C(x),
$$
which, after replacing $x$ by $x^2$, gives the relation
$$
C(-x^2)=2x^2C(x^4)+C(x^2).
$$
Finally, we put the obtained expression for $C(-x^2)$ into (\ref{functrel1}) and obtain the functional relation for $C(x)$ 
in the following form: 
\begin{equation}\label{functrel2}
C(x)+(x-1)C(x^2)-2x^2C(x^4)=0.
\end{equation}

Because $C(x)$ satisfies a Mahler-type functional equation, we know, from the theorem of Nishioka \cite[Theorem~5.1.7]{N}, that $C$ is rational or 
transcendental over $\Q(x)$. Thus, it is enough to check that $C$ cannot represent a rational function.

For the contrary, suppose $C(x)$ is a rational function and not identically zero. Then we can write
$$
C(x) = \frac{P(x)}{Q(x)},
$$
where $P,Q\in\C[x]$, $Q\not\equiv 0$ and $\gcd(P,Q)=1$. Let
$$
p = \deg P,\qquad q = \deg Q,
$$
and let $a$ and $b$ be the leading coefficients of $P$ and $Q$ respectively, i.e.,
$$
P(x) = a x^p + \text{(lower degree terms)},\quad Q(x) = b x^q + \text{(lower degree terms)},
$$
with $a,b\neq 0$.

We put the rational expression for $C(x)$ into (\ref{functrel2}) and clear denominators to obtain the polynomial equality
$$
P(x)Q(x^2)Q(x^4)=(1-x)Q(x)P(x^2)Q(x^4)+2x^2Q(x)Q(x^2)P(x^4).
$$
The degrees of the particular terms (from left to right) are 
$$
p+6q,\qquad 1+2p+5q,\qquad 2+4p+3q.
$$
By comparing the degrees we see that the only possibility for this equality to hold is that $1+2p+5q > 2+4p+3q$. Indeed, if $1+2p+5q < 2+4p+3q$, then $p+6q = 2+4p+3q$ and we get 
$3(q-p)=2$, a contradiction. If $1+2p+5q = 2+4p+3q$, then $2(q-p)=1$, also a contradiction. We thus have $1+2p+5q > 2+4p+3q$, and then we get $p+6q=1+2p+5q$, which implies $q=p+1$.

However, the leading coefficient in $P(x)Q(x^2)Q(x^4)$ is $ab^2$, and the leading coefficient in $(1-x)Q(x)P(x^2)Q(x^4)$ is $-ab^2$, and since the 
common degree of these polynomials is bigger than the degree of $2x^2Q(x)Q(x^2)P(x^4)$, we get that $2ab^2=0$, which is a contradiction.
\end{proof}

\subsection{Divisibility properties}

Another consequence of (the proof of) Theorem \ref{thm:det_H_recurrence} is a simpler recurrence relation describing the parity of $\mathcal{H}_{\mathbf{w}}(n)$ for integer weights. As a special case, we can deduce the periodicity of $\mathcal{H}(n,1) \bmod{2}$. These are Hankel determinants, reduced modulo $2$, for the Thue--Morse sequence $(T_n)_{n \geq 0}$, defined by $T_0=0, T_{2n}=T_n, T_{2n+1}=1-T_n$. A similar result has already been proved for another variant of this sequence, with $T_0=1$ \cite[Proposition 2.2]{APWW}.

 \begin{cor}\label{cor:automaticitymod2}
 Assume that $w_j \in \Z$ for all $j \in \N$.
Let  $n \in \N_+$ and $k \in \N$ be such that $2^k < n \leq 2^{k+1}$. Then
$$\begin{bmatrix}
    \mathcal{H}_{\mathbf{w}}(n) \\
    \widetilde{\mathcal{G}}_{\mathbf{w}}(n)
\end{bmatrix} \equiv 
w_{k+1} \begin{bmatrix}
    1 & w_k \\ 0 & 1
\end{bmatrix}
\begin{bmatrix}
    \mathcal{H}_{\mathbf{w}}(2^{k+1}-n+1) \\
    \widetilde{\mathcal{G}}_{\mathbf{w}}(2^{k+1}-n+1)
\end{bmatrix} \pmod{2}.
$$
In particular, Hankel determinants for the Thue--Morse sequence satisfy
$$ \mathcal{H}(n,1) \equiv n+1 \pmod{2}. $$
 \end{cor}
 \begin{proof}
After reducing modulo $2$, the last summand in \eqref{eq:H_after_reduction} becomes
$$ u_{k+1} \begin{bmatrix}
    0 & 1 \\ 1 & 0
\end{bmatrix} \otimes K_{2^k}.  $$
The remaining calculations proceed exactly as in Theorem \ref{thm:det_H_recurrence}(a). By reducing the formula there modulo $2$, we obtain the first part of our claim.

The second part follows quickly by observing that for weights equal to $1$, the determinants $\widetilde{\mathcal{G}}_{\mathbf{w}}(n)$ are odd for all $n \in \N_+$.
 \end{proof}

It is natural to ask whether the following more general property holds.
\begin{ques}
    Let $k \in \N_+$. Is the sequence $(\mathcal{H}(n+1,1) \bmod{2^k})_{n \geq 0}$  $2$-automatic?
\end{ques}  
Experimental computations suggest that the answer is affirmative at least for $k=2$ and $k=3$, where the automata generating the considered sequences have $12$ and $83$ states, respectively. Moreover, all possible remainders $0,1,\ldots,2^k-1$ seem to appear with similar frequency.

On the other hand, it turns out that almost all determinants $\mathcal{H}(n,1)$ vanish modulo any fixed odd number (in the sense of density). We note that this fact follows from the results of Section \ref{sec:vanishing}.

\begin{cor}
    For any odd $m \in \N_+$ the set $$\{n \geq 1: m \mid \mathcal{H}(n,1) \}$$
    has density $1$.
\end{cor}
\begin{proof}
    Let $d$ be such that $2^d \equiv 1 \pmod{m}$. Then by Theorem \ref{thm:H_zero2} the set $\{ n \geq 1: (t^d - 2^d) \mid \mathcal{H}(n,t)  \}$ has density $1$. The result follows after inserting $t=1$.
\end{proof}

\section{Vanishing and non-vanishing of $\mathcal{H}(n,t)$}\label{sec:vanishing}

In this section, we focus on the determinants $\mathcal{H}(n,t)$ viewed as polynomials in $t$. In particular, we study their vanishing and non-vanishing, which is important for the applications of Hankel determinants, as discussed in Section \ref{sec:Intro}. 

From the results proved so far, we can gather some facts concerning the roots of $\mathcal{H}(n,t)$ for $n \in \N_{\geq 2}$:
\begin{itemize}
\item $\mathcal{H}(n,0)=0 = \mathcal{H}(n,2) $ for $n \geq 3$;
\item $\mathcal{H}(n_k,\zeta)=0$ for any $(k+1)$st root of unity $\zeta$, except $\zeta=1$ (Theorem \ref{thm:nice_formula});
\item the roots of $\mathcal{H}(n,t)$ are algebraic integers (Proposition \ref{prop:degree});
\item if $\mathcal{H}(n,t)=0$ for some $t \in \mathbb{R} \setminus \{0,2\}$, then $t < 0$ (Theorem \ref{thm:factorization}).
\end{itemize}

One may ask the following natural questions:
\begin{enumerate}
    \item What algebraic integers can be roots of $\mathcal{H}(n,t)$ for some $n$?
    \item If $\mathcal{H}(n,t_0)=0$ for some $n \geq 2$ and $t_0 \in \C$, do there exist infinitely many such $n$?
    \item If the above is true, what is the order of the counting function of the set $\{n \geq 2: \mathcal{H}(n,t_0)=0\}$?
\end{enumerate}

\subsection{Behavior at doubled roots of unity}

We begin by exhibiting another infinite family of roots of the polynomials $\mathcal{H}(n,t)$, that is, doubled roots of unity. Instead of relying on the recursive formulas in Theorem \ref{thm:det_H_recurrence}, we take a more direct approach and explicitly construct nonzero vectors lying in $\ker H(n,2\zeta)$ for certain $n$, where $\zeta^d=1$. 

To begin, we give a simple but useful formula for a sum involving $S(n,2\zeta)$.
\begin{lem} \label{lem:alternating_sum}
   Let $d \in \N_+$ and $\zeta^d=1$. Then for each $i=0,1,\ldots, 2^{d}$ we have
   $$  \sum_{j=0}^{2^{d-1}-1} (S(i+2j+1,2\zeta)- S(i+2j,2\zeta)) = 2^{d-1}. $$
\end{lem}
\begin{proof}
If $i$ is even, then each summand equals $1$. If $i$ is odd, say $i=2k+1$, we rewrite the sum as
    \begin{align*} &S(2^d+2k ,2\zeta)-S(2k+1,2\zeta) - \sum_{j=1}^{2^{d-1}-1} (S(2k+2j+1,2\zeta)-S(2k+2j,2\zeta))= \\
    &(2\zeta)^d-1-(2^{d-1}-1)=2^{d-1},
    \end{align*}
    and the lemma is proved.
\end{proof}

We will also need the formula in Proposition \ref{prop:H_recursion}, which in the polynomial case $w_j=t^j$ takes the form
\begin{equation} \label{eq:Ht_recursion}
    H(2^k n,t) = J_{n} \otimes H(2^k,t) + t^k H(n,t) \otimes J_{2^k}  + t^k F(n,t) \otimes K_{2^k},
\end{equation}
where 
$$F(n,t) := G(n,t)-J_{n} = [S(i+j+1,t)-S(i+j,t)-1]_{0\leq i,j < n}.$$
This is precisely the matrix $F_{\mathbf{w}}(n)$, as in Section \ref{sec:H} for the weights $w_j = t^j$.

We now show how to construct vectors in $\ker H(2^d n,2\zeta)$ for a $d$th root of unity~$\zeta$. When writing the product $\otimes$ of vectors or subspaces, we identify the tensor and Kronecker product.

\begin{lem} \label{lem:kerFH}
Let $d,n \in \N_+$ and $\zeta^d=1$. Then
$$\ker F(n,2\zeta) \otimes \ker H(2^d,2\zeta) \subset \ker H(2^d n,2\zeta).$$
\end{lem}
\begin{proof}

If $U \in \ker F(n,2\zeta)$ and $V \in \ker H(2^d,2\zeta)$, then \eqref{eq:Ht_recursion} implies
$$H(2^d n,2\zeta)(U \otimes V) = t^d H(n,2\zeta) U \otimes J_{2^d}V.$$
We now argue that $J_{2^d}V=\mathbf{0}$. Since the kernel and image of a symmetric matrix are orthogonal, it suffices to show that the vector $\mathbf{1}_{2^d}$ is in the range of $H(2^d,2\zeta)$.  But the $i$th component of
$$  H(2^d,2\zeta) \left(\mathbf{1}_{2^{d-1}} \otimes \begin{bmatrix}
     -1 \\ 1
 \end{bmatrix} \right) $$
 is precisely the sum in Lemma \ref{lem:alternating_sum}, and the result follows.
\end{proof}

The next lemma gives a collection of vectors which will serve as ``building blocks'' for Lemma \ref{lem:kerFH}.

\begin{lem} \label{lem:kerFH2}
    Let $d, k \in \N_+$, $\zeta^d=1$, and let $V \in \C^{2^k}$ be any column vector. Then
    \begin{enumerate}[label=(\alph*)]
    \item $\mathbf{1}_{2^{d-1}} \otimes [1,0]^T,\mathbf{1}_{2^{d-1}} \otimes [0,1]^T \in \ker F(2^d,2 \zeta)$;
    \item if $\mathbf{1}_{2^{k}}^T V=0$, we have $\mathbf{1}_{2^{d-1}} \otimes [1,-1]^T \otimes V \in \ker H(2^{d+k},2 \zeta)$;
    \item $\mathbf{1}_{2^{d-1}}\otimes(2[1,0]^T\otimes\mathbf{1}_{2^{d-1}}\otimes[1,-1]^T-[1,-1]^T\otimes[1,\mathbf{0}_{2^{d}-1}]^T)\in \ker H(2^{2d},2 \zeta)$.
\end{enumerate}
\end{lem}
\begin{proof}
Put $t = 2 \zeta$.
    Starting with (a), the $i$th component of $F(2^d,t) \left(\mathbf{1}_{2^{d-1}} \otimes [1,0]^T\right)$ is
    $$ \sum_{j=0}^{2^{d-1}-1} \left( S(i+2j+1,t) - S(i+2j,t) - 1 \right)=0,   $$
    by Lemma \ref{lem:alternating_sum}. 
    In the case of $\mathbf{1}_{2^{d-1}} \otimes [0,1]^T$, we need to replace $i$ with $i+1$ above, and Lemma \ref{lem:alternating_sum} again implies our claim.

    For the sake of (b) and (c) we put $W = \mathbf{1}_{2^{d-1}} \otimes [1,-1]^T$. To prove (b), we apply the equality \eqref{eq:Ht_recursion} to $n=2^d$. Using (a) and $J_{2^d} W =\mathbf{0}_{2^d}, J_{2^k} V =\mathbf{0}_{2^k}$, we quickly obtain $  H(2^{d+k},t)(W \otimes V) = \mathbf{0}_{2^{d+k}}.$
 
    In (c) we again apply \eqref{eq:Ht_recursion} to $k=d, n=2^d$, obtaining
    $$  H(2^{2d},t) (\mathbf{1}_{2^{d-1}}\otimes[1,0]^T\otimes W) = 2^{d-1} \mathbf{1}_{2^{d}} \otimes H(2^d,t)W = -2^{2d-2}  \mathbf{1}_{2^{2d}}.$$
   At the same time, by Lemma \ref{lem:alternating_sum} we get
    $$  H(2^{2d},t) \left(W \otimes\begin{bmatrix}
        1 \\ \mathbf{0}_{2^{d}-1}
    \end{bmatrix}\right) = t^d H(2^d,t) W \otimes J_{2^d} \begin{bmatrix}
        1 \\ \mathbf{0}_{2^{d}-1}
    \end{bmatrix} = -2^{2d-1} \mathbf{1}_{2^{2d}}. $$
    Both equalities together give (c).
\end{proof}

We are now ready to state a result which says that for any root of unity $\zeta$ the sequence $(\mathcal{H}(n,2\zeta))_{n \in \N}$ contains arbitrarily long runs of zeros at indices concentrated around multiples of powers of $2$.

\begin{thm} \label{thm:H_zero}
Let $\zeta \in \C$ be a root of unity of order $d \geq 2$ and let $l,s \in \N$ be such that $l \geq d+1$ and $s$ is odd. Write $l = qd +r$, where $q \geq 1, r \in\{1,\ldots,d\},$ and put
$$z_{d,l} = 2^r \frac{2^{qd}-1}{2^d-1}-\begin{cases}
        2 &\text{if } r<d, \\
        1 &\text{if } r=d.
    \end{cases}$$
Then for $i=-z_{d,l}, \ldots, z_{d,l}+1$ we have
$$\mathcal{H}(2^l s + i,2\zeta)=0.$$
In particular, we have $\det(H(n,-2))=0$ for all $n \neq n_k-1,n_k, n_{k}+1$, where the sequence $(n_k)_{k \geq 0}$ is defined by \eqref{eq:n_k}.
\end{thm}
\begin{proof}
    We first find for each $l \geq d+1$ a nonzero vector $U_l \in \ker H(2^l,2\zeta)$ with possibly many trailing zeros.

    For $l=d+1, \ldots,2d-1$ by Lemma \ref{lem:kerFH2}(b) we take 
    $$U_l = \mathbf{1}_{2^{d-1}} \otimes \begin{bmatrix}
        1 \\ -1
    \end{bmatrix}
    \otimes \begin{bmatrix}
        1 \\ -1 \\ \mathbf{0}_{2^{l-d}-2}
    \end{bmatrix}.$$
    We can use the same formula for any $l$, however in order to obtain more trailing zeros we put
    $$  U_{2d} = \mathbf{1}_{2^{d-1}}\otimes \left(2\begin{bmatrix}
      1 \\ 0  
    \end{bmatrix}
    \otimes\mathbf{1}_{2^{d-1}}\otimes
    \begin{bmatrix}
      1 \\ -1  
    \end{bmatrix}-\begin{bmatrix}
      1 \\ -1  
    \end{bmatrix}\otimes\begin{bmatrix}
      1 \\ \mathbf{0}_{2^{d}-1}  
    \end{bmatrix}\right),  $$
    as in Lemma \ref{lem:kerFH2}(c). For $l \geq 2d+1$ we recursively define
    $$U_l = \mathbf{1}_{2^{d-1}} \otimes \begin{bmatrix}
        1 \\ 0
    \end{bmatrix} \otimes U_{l-d}$$
    and by Lemmas \ref{lem:kerFH} and \ref{lem:kerFH2}(a) we have $U_l \in \ker H(2^l,2\zeta)$.
    
    Let $z(U)$ denote the number of trailing zeros in a vector $U$. We have 
    $$z(U_l)=2^{l-d} -\begin{cases}
        2 &\text{if } d+1 \leq l < 2d, \\
        1 &\text{if } l=2d,
    \end{cases}$$
    and
    $$
    z(U_l) = 2^{l-d}+z(U_{l-d}), \qquad l > 2d.
    $$
    It is straightforward to check that $z(U_l)=z_{d,l}$.
    
    Now, let $s$ be odd and note that the matrix $F(s,t)$ is singular for any $t$, since its entries at positions $(i,j)$ with $i \equiv j \pmod{2}$ are zero.
    Choose any nonzero $V_s \in \ker F(s,2\zeta)$, so that
    $$ U_{l,s} := V_s \otimes U_l \in \ker H(2^l s,2 \zeta).$$
    Observe that $z(U_{l,s}) \geq z(U_l) = z_{d,l}.$
    
    We now check that for $-z_{d,l} \leq i \leq z_{d,l}+1$ we have $\mathcal{H}(2^ls+i,2\zeta)=0$. For $i \leq 0$ we truncate $|i|$ zeros from $U_{l,s}$, obtaining a nonzero vector in  $\ker H(2^ls-|i|,2\zeta)$. For $i > 0$ we can append $i$ zeros to $U_{l,s}$, again yielding an element of $\ker H(2^ls+i,2\zeta)$. To see why this is true, consider the vector
$$ U_{l,s}' = \begin{bmatrix}
    U_{l,s} \\ 0_{2^{l}}
\end{bmatrix} = \begin{bmatrix}
    V_s \\ 0
\end{bmatrix} \otimes U_l. $$
By equality \eqref{eq:Ht_recursion} we get
$$ H(2^{l}(s+1),2\zeta) \:U_{l,s}' = (2\zeta)^l \left( F(s+1,2\zeta)\begin{bmatrix}
    V_s \\ 0
\end{bmatrix}\right)  \otimes  \left( K_{2^l}U_l\right) = \begin{bmatrix}
    0_s \\ C
\end{bmatrix} \otimes  \left( K_{2^l}U_l\right) $$
for some $C \in \C$.
The product $K_{2^l}U_l$ starts with with $z_{d,l}+1$ zeros. Hence, the whole vector starts with $2^l s+z_{d,l}+1$ zeros so $$H(2^{l}s+i,2\zeta) \begin{bmatrix}
    U_{l,s} \\ \mathbf{0}_i
\end{bmatrix} = \mathbf{0}_{2^l s+i}$$
for $i=1,\ldots,z_{d,l}+1$, and the first part of our claim follows.

In the case $d=2$, we have
$$z_{2,l} = \begin{cases}
  \frac{2^l-7}{3} &\text{if } 2\mid l, \\
   \frac{2^l-8}{3} &\text{if } 2\nmid l.
\end{cases}$$
If $l \geq 3$ is even, then $\mathcal{H}(n,-2)=0$ for $2^l-z_{2,l} \leq n \leq 2^l+z_{2,l}+1$, where
\begin{align*}
  2^l-z_{2,l}  &= \frac{2^{l+1}+7}{3} = n_{l-1}+2, \\
  2^l+z_{2,l}+1  &= \frac{2^{l+2}-4}{3} = n_{l}-2
\end{align*}
We get the same equalities for $l$ odd, which proves the second part of the result.
\end{proof}
For each fixed $d \geq 2$ consider the set of indices as in Theorem \ref{thm:H_zero}:
$$A_d := \N_+ \cap \bigcup_{l=d+1}^{\infty} \bigcup_{\substack{s \in \N \\ s \text{ odd}}} (2^ls-z_{d,l}-1,2^ls+z_{d,l}+1].$$
In particular, we have
$$ A_2 =  \N_+ \setminus \bigcup_{k=0}^{\infty} \{n_k-1,n_k,n_k+1 \}. $$
\bigskip

In the following theorem we bound from above the counting function of the complement of $A_d$. In particular, the set of $n$ such that $H(n,2\zeta) \neq 0$ turns out to have density $0$.

\begin{thm}\label{thm:H_zero2}
 For real $x \geq 1$ we have
 $$  \# \{n \leq x:   H(n,-2) \neq 0 \} \leq 3 \log_2 x. $$
 If $d \geq 3$ and $\zeta^d=1$, then
 $$\# \{n \leq x:   H(n,2\zeta) \neq 0 \} =O(x^{\delta_d}),  \vspace{-3pt}$$ 
 where $\delta_d < 1 + \log_2(1-2^{-d}) < 1$.
\end{thm}
\begin{proof}
The claim for $t=-2$ follows directly from the second part of Theorem \ref{thm:H_zero}, hence fix $d \geq 3$.
Instead of $A_d$, we consider the simpler set
$$\widetilde{A}_d:=\bigcup_{l=2d}^{\infty} \bigcup_{\substack{s \in \N \\ s \text{ odd}}} (2^ls-2^{l-d},2^ls+2^{l-d}],  $$
which satisfies $\N_+ \cap \widetilde{A}_d \subset A$.

For  $k \in \N$ let $\widetilde{A}_{d,k} :=\widetilde{A}_d \cap (0,2^k]$. 
We are going to find a recurrence relation for the cardinalities $a_k := \# (\widetilde{A}_{d,k} \cap \N_+)$, which can be conveniently expressed as the Lebesgue measure $\lambda(\widetilde{A}_{d,k})$.
For $k \geq 3d$ we have 
$$  \widetilde{A}_{d,k} = \widetilde{A}_{d,k-1} \cup (2^k - \widetilde{A}_{d,k-1}) \cup (2^k-2^{k-d}, 2^k],$$
up to a finite set (consisting of endpoints of the intervals comprising $2^k - \widetilde{A}_{d,k-1}$).
The second and third set have a nontrivial intersection of Lebesgue measure
$$ \lambda((2^k - \widetilde{A}_{d,k-1}) \cap (2^k-2^{k-d}, 2^k]) = \lambda(-\widetilde{A}_{d,k-1}\cap (-2^{k-d},0]) = a_{k-d}.$$
Hence, for $k \geq 3d$ we have the recurrence relation
$$ a_k = 2a_{k-1}+2^{k-d}-a_{k-d}.$$
Letting $b_k = 2^k-a_k = \lambda((0,2^k]\setminus \widetilde{A}_{d,k})$, this can be rewritten as
$$ b_k = 2b_{k-1}-b_{k-d}. $$
The characteristic polynomial of this recurrence relation is
$$ B(x) = x^d - 2x^{d-1} + 1 = (x-1)(x^{d-1} - x^{d-2} - \cdots - x - 1). $$
Now, \cite[Lemma 3.6]{W}  implies that the second factor has a unique real root $R$ lying in the interval $(1, 2)$, while all other complex roots have absolute value less than $1$. Moreover, we have $R < 2-2^{1-d}$, since    $B$ is increasing in the interval $(\frac{2(d-1)}{d},+\infty)$ containing $2-2^{1-d}$, and $B(2-2^{1-d})>0$.
Letting $\delta_d = \log_2 R$, we get $b_k = O(2^{k \delta})$.
To finish the proof, let $2^{k-1} \leq x < 2^k$ so that $k= \log_2 x + O(1)$. Then
$$ \# \{n \leq x:   H(n,2\zeta) \neq 0 \} \leq b_k = O(x^{\delta_d}). \qedhere$$
\end{proof}

\begin{rem}
Since $\delta_d<1$, the theorem implies in particular that the set of indices $n\leq x$ for which $H(n,2\zeta)\neq 0$ has sublinear growth. 
\end{rem}

We implemented the recurrence from Proposition~\ref{prop:h_g_recurrence} in Mathematica~\cite{Mathematica} to test the vanishing pattern predicted by Theorem~\ref{thm:H_zero}.  In particular, for all $2 \leq n\leq 1000$ and for primitive roots of unity $\zeta$ of orders $d=2,3,4,5,7$, the set of indices such that $\mathcal{H}(n,2\zeta) = 0$ is precisely $A_d\cap \{1,\ldots,1000\}$. This suggests the following conjecture. 

\begin{conj}\label{conj:Ad}
Let $d \in \N_{\geq 2}$, and let $\zeta$ be a primitive $d$th root of unity.  Then for
every $n \in  \N_{\geq 2}$, we have
$$
    \mathcal{H}(n,2\zeta)=0
$$
if and only if $n\in A_d$. 

In particular, $\mathcal{H}(n,-2)=0$ if and only if $n \neq n_k-1,n_k,n_k+1$.

\end{conj}

The ``hard part'' of the conjecture is to prove that the determinants $\mathcal{H}(n,2\zeta)$ do not vanish outside of $A_d$. This is already known in the special case $n_k = \lceil 2^{k+2}/3 \rceil$, where the formula from Theorem \ref{thm:nice_formula} shows that the only roots of $\mathcal{H}(n_k,t)$ are $0,2$ and $(k+1)$st roots of unity. However, such explicit characterizations seem rare. We now illustrate a possible different approach, which does not rely on computing the determinants directly.

\begin{lem} \label{lem:H_nonsingular}
Let $n \in \N_+$ and assume that $\mathbf{1}_{n}$ is in the image of $H_\mathbf{w}(n)$. If $\mathcal{G}_\mathbf{w}(n) \neq 0$, then $\mathcal{H}_\mathbf{w}(n) \neq 0$.
\end{lem}
\begin{proof}
Let $X \in \C^n$ be such that $H_\mathbf{w}(n) X = \mathbf{1}_{n}.$ For $m \geq 0$ consider row vectors $R_m = [s_\mathbf{w}(n), \ldots, s_\mathbf{w}(m+n-1)]$. In particular, for $m < n$ this is the $m$th row of $H_\mathbf{w}(n)$.
Letting $I_n$ denote the $n \times n$ identity matrix, we have
$$ H_\mathbf{w}(n)(X R_n-I_n) = \mathbf{1}_n R_n- H_\mathbf{w}(n) = \begin{bmatrix}
    R_n-R_0\\
    R_n-R_1\\
    \vdots \\
    R_n-R_{n-1}
\end{bmatrix} = \begin{bmatrix}
    1 &\cdots &  \cdots & 1 \\
    0 & \ddots &  & \vdots \\
    \vdots & \ddots & \ddots & \vdots \\
    0 & \cdots & 0 & 1
\end{bmatrix} G_\mathbf{w}(n),  $$
and the result follows.
\end{proof}

\begin{rem}
From the proof, we can derive the relation
$$(-1)^{n-1}(R_n X-1)\mathcal{H}_\mathbf{w}(n) = \mathcal{G}_\mathbf{w}(n). $$
Indeed, the eigenvalues of $X R_n$ are: $0$ with multiplicity $n-1$, and $R_n X$. Hence, $\det (X R_n-I_n) = (-1)^{n-1}(R_n X-1)$, and the claim follows.
\end{rem}

We now apply the lemma to exhibit a family of indices different from $n_k$, where $\mathcal{H}(n,2\zeta) \neq 0$.

\begin{thm} \label{thm:H_nonzero}
Let $d \in  \N_{\geq 2}$ and let $\zeta$ be a primitive root of unity of order $d$. Then for any $q \geq 1 , r \in \{0,1,\ldots,d-1\}$,
and $$ n = 2^{qd+r}+2^{(q-1)d+r}+\cdots + 2^{d+r} + 2^r = 2^r \frac{2^{(q+1)d}-1}{2^d-1}, $$
we have
$$  \mathcal{H}(n,2\zeta) \neq 0. $$
\end{thm}
\begin{proof}
First, note that the binary expansion of $n$ is $(10^{d-1})^k10^r$, which means that $\mathcal{G}(n,2\zeta)\neq 0$ by Corollary  \ref{cor:det_G(n,t)}. By Lemma \ref{lem:H_nonsingular} it remains to show that $\mathbf{1}_n$ lies in the image of $H(n,2\zeta)$.

We first construct a sequence of vectors $(V_l)_{l \geq d}$ such that $H(2^l,2\zeta)V_l = \alpha_l \mathbf{1}_{2^d}$
for some $\alpha_l \in \C$. For $l=d$, we may take
$$V_d = \mathbf{1_{2^{d-1}}} \otimes \begin{bmatrix}
    1 \\ -1
\end{bmatrix},$$
as in the proof of Lemma \ref{lem:kerFH}.
Further, for $l=d+1,\ldots,2d-1$ we put
$$ V_{l} = V_d \otimes \begin{bmatrix}
    1 \\ 0_{2^{l-d}-1}
\end{bmatrix}, $$
which is a valid choice thanks to the equality \eqref{eq:Ht_recursion} applied to $k=l-d, n= 2^d$. Finally, for $l \geq 2d$ we let
$$V_l = \left( \mathbf{1}_{2^{d-1}} \otimes \begin{bmatrix}
    1 \\0
\end{bmatrix}\right) \otimes V_{l-d},$$
where again we use \eqref{eq:Ht_recursion} and Lemma \ref{lem:kerFH2}(a).

We now count the trailing zeros of $V_l$. For $l=d,\ldots, 2l-1$ we have $z(V_l) = 2^{l-d}-1$, whereas for $l \geq 2d$ we get
$$  z(V_l) = 2^{l-d} + z(V_{l-d}). $$
Writing $l=qd+r$ as in the statement, we thus get $z(V_l) = n - 2^l -1$. A similar computation as in the proof of Theorem \ref{thm:H_zero}  shows that
$$ H(2^{l+1},2\zeta) \left(\begin{bmatrix}
    1 \\ 0
\end{bmatrix} \otimes V_l\right) = \begin{bmatrix}
    1 \\ 1
\end{bmatrix} \otimes \alpha_l \mathbf{1}_{2^l} + (2 \zeta)^l \begin{bmatrix}
    0 \\ 2 \zeta-2
\end{bmatrix}  \otimes K_{2^l} V_l. $$
But $K_{2^l} V_l$ has $z(V_l)+1$ leading zeros so the top $2^l+z(V_l)+1=n$ entries in the whole vector are $\alpha_l$. It follows that 
$$H(n,2\zeta)\begin{bmatrix}
    V_l \\ \mathbf{0}_{n-2^l}
\end{bmatrix}= \alpha_l \mathbf{1}_n,$$ as desired.
\end{proof}

\subsection{Vanishing of $\mathcal{H}(n,t)$ for $t \neq 0, 2\zeta$}
For other algebraic numbers $t$ it seems much harder to describe the set of $n$ such that $\mathcal{H}(n,t)=0$. However, as shown in the following result and its corollary, once it is known that $\mathcal{H}(n,t)=0$ for some $n \geq 2$, it is possible to construct infinitely many such indices.

\begin{thm} \label{thm:H_p_q_zero}
Let $t \in \C$ be such that $t \neq 0$ and $t/2$ is not a root of unity. Let $p,q \in \N_+$ be such that $p\geq 3$ is odd and $\mathcal{H}(p,t) = \mathcal{H}(q,t) = 0$. 
Then 
$$\mathcal{H}(2^{m}(p-1)+q,t)=0,$$
where $m \in \N_+$ is such that $2^{m-1} < q \leq 2^m$.
\end{thm}
\begin{proof}
Let $n \in \N$ be such that $p \leq 2^n$.
    There exist nonzero column vectors $U' \in \C^p$ and $W' \in \C^q$ such that $H(p,t) U' = 0$ and $H(q,t)W' = 0$. We pad $U'$ and $W'$ with trailing zeros, obtaining $U \in \C^{2^n}$ and $W \in \C^{2^m}$. 

    For a vector $V \in \C^k$ we will now write $\Sigma V = \mathbf{1}_k^T V$ to denote the sum of its components.
    Let $V \in \C^{2^n}$ and $X \in \C^{2^m}$ be such that
\begin{equation} \label{eq:UVWX}
\begin{aligned} 
        H(2^n,t) U &= t^n K_{2^n} V, \\
        H(2^m,t) W &= t^m K_{2^m} X, \\
        \Sigma U &= \Sigma V, \\
        \Sigma W  &= \Sigma X.
\end{aligned}
\end{equation}
This is possible since the first row of the matrices $K_l$ is zero and the remaining rows are linearly independent.

Our goal now is to construct two vectors $Y,Z \in \C^{2^{m+n}}$ satisfying 
\begin{equation} \label{eq:HY_KZ}
    H(2^{m+n},t)Y = t^{m+n} K_{2^{m+n}}Z,
\end{equation}
such that $Y$ ends with and $Z$ begins with sufficiently many zeros, so that after truncating we get a nonzero vector $Y' \in \operatorname{ker} H(2^{m}(p-1)+q,t)$. They will be of the form
 \begin{align*}
     Y &= A \otimes W+ B \otimes X, \\
     Z &= C \otimes W + D \otimes X,
 \end{align*}
for some $A,B,C,D \in \C^{2^n}$ satisfying
\begin{equation} \label{eq:ABCD}
\begin{aligned} 
    A+B &= U, \\
    C+D &= V, \\
    \Sigma B &=0, \\
    F(2^n,t) A &= t^n L_{2^n} C.
\end{aligned}
\end{equation}

We first choose $A$ so that it ends with many zeros, $C$ resulting from \eqref{eq:ABCD} begins with many zeros, and  $\Sigma A = \Sigma U$.
  We claim that it is possible to choose a vector $A'$ (which will constitute the first $p$ components of $A$) from the kernel of $F(p,t)$, in such a way that $\Sigma A' = \Sigma U$. If $\Sigma U = 0$, we simply take $A' = \mathbf{0}_p$. Otherwise, if $\Sigma U \neq 0$, we use the fact that the matrix $F(p,t)$ is singular for $p$ odd. Suppose by contradiction that $\Sigma A' = 0$ for all $A' \in \operatorname{ker} F(p,t)$. Because the matrix $F(p,t)$ is symmetric, $\mathbf{1}_p$ must lie in its image, and thus  $J_p = F(p,t) Q$ for a $p \times p$ matrix $Q$. Consequently, the matrix $G(p,t) = F(p,t) + J_p$ is also singular. But this cannot be the case, as by Corollary \ref{cor:det_G(n,t)} the roots of $\mathcal{G}(p,t)$ are $t=0$ and $t=2\zeta$, where $\zeta$ is a root of unity.  Hence, there exists $A' \in \operatorname{ker} F(p,t)$ wih $\Sigma A'$ nonzero and after scaling we can assume $\Sigma A' = \Sigma U$. Then $A \in \C^{2^n}$ is obtained by padding $A'$ with $2^n-p$ trailing zeros, so that $\Sigma A = \Sigma U$ holds. The remaining vectors $B,C,D$ are uniquely determined by the imposed conditions \eqref{eq:ABCD}.

  Next, we show that the relations \eqref{eq:UVWX} and \eqref{eq:ABCD} imply equality \eqref{eq:HY_KZ}. We expand $H(2^{m+n},t)Y$ using Proposition \ref{prop:H_recursion}, obtaining three products on the left-hand side. The first is
\begin{align*}
    (J_{2^n} \otimes H(2^m,t) ) Y &= J_{2^n} A \otimes H(2^m,t) W + J_{2^n} B \otimes H(2^m,t) X \\
    &= t^m (\Sigma U) \mathbf{1}_{2^n} \otimes K_{2^m} X.
\end{align*}  
The second product is
\begin{align*}
    (t^m H(2^n,t) \otimes J_{2^m} ) Y &= t^m H(2^n,t) A \otimes J_{2^m} W + t^m H(2^n,t) B \otimes J_{2^m} X \\
    &= t^{m+n} (\Sigma W) K_{2^n} V \otimes \mathbf{1}_{2^m}.
\end{align*} 
The final product on the left-hand side is
\begin{align*}
t^m (F(2^n,t)\otimes K_{2^m})) Y&=
    t^m F(2^n,t) A \otimes K_{2^m}W + t^m F(2^n,t) B \otimes K_{2^m}X \\ 
    &=t^{m+n}L_{2^n} C \otimes K_{2^m}W +  t^m F(2^n,t) B \otimes K_{2^m}X
\end{align*} 
The right-hand side of \eqref{eq:HY_KZ} is
\begin{align*}
   &t^{m+n}\bigl(L_{2^n}\otimes K_{2^m}+K_{2^n} \otimes J_{2^m}\bigr)Z = \\
   &t^{m+n} \bigl(L_{2^n}C \otimes K_{2^m}W + L_{2^n}D \otimes K_{2^m}X + (\Sigma W) K_{2^n}V \otimes \mathbf{1}_{2^m}\bigr).
\end{align*} 
In order to prove our claim it remains to show that
$$F(2^n,t) B + \mathbf{1}_{2^n} = t^n L_{2^n} D.  $$
After adding $F(2^n,t) A = t^n L_{2^n} C$ to both sides and multiplying by $L_{2^n}$ from the left, this is equivalent to
\begin{equation} \label{eq:V_form}
    t^n  V = L_{2^n} F(2^n,t) U + (\Sigma U) \mathbf{1}_{2^n}.
\end{equation}
Note that, given $U$, the conditions $H(2^n,t) U  = t^n K_{2^n} V$ and $\Sigma U = \Sigma V$ determine $V$ uniquely. Hence, it is enough to check that $V$ as in \eqref{eq:V_form} satisfies them both. The first one follows from the identity
$$ H(2^n,t) = K_{2^n}( L_{2^n} F(2^n,t) + J_{2^n}) + \mathbf{1}_{2^n} \begin{bmatrix}
    S(0,t) & \cdots & S(2^n-1,t)
\end{bmatrix},$$
which in turn is a direct consequence of the definition of $F$.
The second condition is obtained by summing the components of both sides of \eqref{eq:V_form} and using the equality
$$ \mathbf{1}_{2^n}^T F(2^n,t) = (t^n - 2^n) \mathbf{1}_{2^n}^T.  $$

We now count the trailing zeros in $Y$ and the leading zeros in $Z$, where the given amounts are not necessarily maximal.
Both $A$ and $U$ end with $2^n-p$ zeros, hence so does $B=U-A$. It follows that $Y$ ends with $2^n-p$ blocks of zeros having length $2^m$. Additionally, the $(p-1)$st block in $Y$ (numbering from $0$) is a linear combination of $W$ and $X$, and these vectors end with $2^m-q$ and $q-1$ zeros, respectively. Since $q-1 > 2^m-q$, we deduce that $Y$ ends with $2^m(2^n-p) + 2^m-q = 2^{m+n} - (2^{m}(p-1)+q)$ zeros. 
Furthermore, from the relations $F(p,t)A'=\mathbf{0}_p$ and $F(2^n,t) A = t^n L_{2^n} C$ we deduce that $C$ ends with  $p$ zeros. At the same time, $V$ ends with $p-1$ zeros, which means that $D = V-C$ also ends with $p-1$ zeros. Hence, $Z$ ends with $p-1$ blocks of zeros of length $2^m$ and the previous block is a multiple of $X$. Therefore, $Z$ ends with at least $2^m(p-1) + q-1$ zeros. If we let $Y'$ denote the vector $Y$ truncated to the initial $2^{m}(p-1)+q$ components, then by the above considerations \eqref{eq:HY_KZ} implies
$$H(2^{m}(p-1)+q,t) Y' = \mathbf{0}_{2^{m}(p-1)+q}.$$

It remains to show that $Y'$, or equivalently, $Y$ is nonzero. Suppose by contradiction that $Y = A \otimes W+ B \otimes X$ is the zero vector. 

Since $A+B=U \neq \mathbf{0}_{2^n}$, at least one of $A,B$ is non-zero, and thus $X,W$ must be linearly dependent. Because $W$ is non-zero, we can write $X = \alpha W$  for some $\alpha \in \C$, which gives $Y = (A + \alpha B) \otimes W$, and consequently $A + \alpha B = \mathbf{0}_{2^n}$. Then $\Sigma U = 0$, in which case we have chosen $A = \mathbf{0}_{2^n}$ in an earlier step in the proof. It follows that $\alpha=0$, meaning that $X = \mathbf{0}_{2^m}$. By the conditions \eqref{eq:UVWX}, we have $H(2^m,t)W = \mathbf{0}_{2^m} $ and $\Sigma W =\Sigma X = 0$. Since the matrix $H(2^m,t)$ is symmetric, $\mathbf{1}_{2^m}$ belongs to its image. But $\mathcal{G}(n,t) \neq 0$ by Corollary \ref{cor:det_G(n,t)} so Lemma \ref{lem:H_nonsingular} implies $H(2^m,t)$ is non-singular, a contradiction.
\end{proof}

If we take $q=p$ in the theorem, we can construct a sequence of indices such that $\mathcal{H}(p_k,t)=0$.

\begin{cor}\label{cor:nonunity2}
Let $t \in \C$ be such that $t \neq 0$ and $t/2$ is not a root of unity. Let $p \geq 3$ be odd and such that $\mathcal{H}(p,t) = 0$. 
Then for all $k \in \N$ we have
$$\mathcal{H}(p_k,t)=0,$$
where
$$ p_k =  (p-1) \frac{2^{(k+1)m}-1}{2^m-1} +1$$
and $m \in \N_+$ is such that $2^{m-1} < p \leq 2^m$.
\end{cor}
\begin{proof}
By Theorem \ref{thm:H_p_q_zero} we have $\mathcal{H}(p'_k,t)=0$, where 
$p'_0 = p$ and $p'_{k+1} = 2^m(p'_k-1)+p$. It is easy to check by induction that $p'_k = p_k$.
\end{proof}

For example, if we take $t = -1$ and $p=3,2^m=4$, then we retrieve the subsequence $p_k = n_{2k+1}$, where $n_k$ is defined by \eqref{eq:n_k}.

\section{Further questions and problems} \label{sec:Questions}

We conclude with several additional questions and problems that are naturally suggested by the results of this paper.

\begin{prob}
Let $t\in\C$ be fixed and such that $\mathcal{H}(n,t)=0$ for some $n \in \N_{\geq 2}$. Describe the structure of the set
\[
\{n\ge 0 : \mathcal{H}(n,t)=0\}.
\]
In particular, is this set automatic?
\end{prob}

In view of the polynomial formula in Theorem \ref{thm:nice_formula} and the vanishing patterns obtained in Section \ref{sec:vanishing}, this is an interesting question especially in the case when $t =\zeta$ and $t=2\zeta$, where $\zeta$ is a root of unity.

\begin{ques}
Do there exist $n\ge 2$ and $t\in\C$ such that $\mathcal{H}(n,t)$ has a multiple root
at $t$, other than $t=0$ and $t=2\zeta$?
\end{ques}

Another interesting problem is further restricting the set of possible roots of the polynomials $(\mathcal{H}(n,t))_{n \geq 2}$. We state two further questions in this direction.

\begin{ques}
Does there exist a uniform bound $|t| \leq M$ satisfied by all roots of $\mathcal{H}(n,t)$ for all $n \geq 2$?
\end{ques}
Experimental computations suggest that the answer is affirmative and the optimal bound is $M=\frac{\sqrt{5}+3}{2}$. Moreover, the value $-\frac{\sqrt{5}+3}{2}$ can apparently be approached from above by real roots of the polynomials $\mathcal{H}(a_k,t)$, where $a_0=5$ and
$$  a_{k+1} = \begin{cases}
    2 a_k &\text{if } k \equiv 0,1 \pmod{4}, \\
    2 a_k-1 &\text{if } k \equiv 2,3 \pmod{4}.
\end{cases} $$

\begin{ques}
Is it true that for every fixed degree $D\ge 1$ there exist only finitely many
algebraic numbers $t\in\C$ of degree at most $D$ such that $\mathcal{H}(n,t)=0$ for
some $n\ge 2$?
\end{ques}

Equivalently, do the polynomials $\mathcal{H}(n,t)$ have only finitely many distinct
irreducible factors of any fixed degree over $\Q$? Our computations suggest that this is indeed the case.

\begin{prob}
Find further natural families of indices $n$ for which $\mathcal{H}(n,t)$ admits a
simple closed formula (for general $t$ as well as for fixed $t=t_{0})$.
\end{prob}

Besides Theorem \ref{thm:nice_formula}, the formulas from Proposition \ref{prop:special_cases} suggest that other such families may exist.

\begin{prob}
Describe more precisely the set of non-vanishing indices for $t=2\zeta$, where
$\zeta$ is a root of unity.
\end{prob}

Theorems \ref{thm:H_zero} and \ref{thm:H_zero2} show that this set is sparse, but its finer structure remains unclear.

A natural direction of further research is to consider more general sequences than those considered in the paper. In particular, the methods of the paper are strongly tied to the binary expansion, but it is natural to ask to what extent they persist in other bases. More precisely, we formulate the following problem.

\begin{prob}
Extend the results of this paper to weighted sum-of-digits functions in base
$b\ge 2$.
\end{prob}

One may also ask for the values of Hankel determinants for a shift of the sequence $(s_{\mathbf{w}}(n))_{n \geq 0}$, as was done in \cite{APWW} in the Thue--Morse case.

\begin{prob}
    Find recursive or explicit formulas for Hankel determinants corresponding to the shifted sequence $(s_{\mathbf{w}}(n+k))_{n \geq 0}$, where $k \in \N$.
\end{prob}

Finally, we formulate the following general problem.

\begin{prob}
Investigate Hankel determinants associated with other regular digital sequences, for instance sequences counting occurrences of a fixed block in the base-$b$ expansion of $n$.
\end{prob}

These sequences share many structural features with weighted sum-of-digits functions and may exhibit similar determinant phenomena.

\bibliographystyle{plain}

\bigskip

\noindent  Bartosz Sobolewski, Jagiellonian University, Faculty of Mathematics and Computer Science, Institute of Mathematics, {\L}ojasiewicza 6, 30 - 348 Krak\'{o}w, Poland\\
e-mail:\;{\tt  bartosz.sobolewski@uj.edu.pl} \vspace{1em}

\noindent  Maciej Ulas, Jagiellonian University, Faculty of Mathematics and Computer Science, Institute of Mathematics, {\L}ojasiewicza 6, 30 - 348 Krak\'{o}w, Poland\\
e-mail:\;{\tt  maciej.ulas@uj.edu.pl}
\end{document}